\documentclass{amsart}
\usepackage{amssymb, amsmath}
\usepackage{array}
\usepackage{xcolor}

\usepackage{comment}
\usepackage{enumerate}

\newtheorem{Theorem}{Theorem}
\newtheorem{Definition}[Theorem]{Definition}

\newtheorem{Lemma}[Theorem]{Lemma}
\newtheorem{Corollary}[Theorem]{Corollary}

\newcommand{\sgn}{\mbox{sgn}}

\DeclareMathOperator{\Id}{Id}

\newcommand{\I}{\mbox{Id}}

\def\gp#1{\langle#1\rangle}

\begin{document}
\title[Upper triangular matrices with superinvolution]{upper triangular matrices with superinvolution: \\ identities and images of multilinear polynomials}

\author{Elena Campedel}
\address{Dipartimento di Ingegneria e Scienze dell'Informazione e Matematica, Universit\`a degli Studi dell'Aquila, Via Vetoio 1, 67100, L'Aquila, Italy}
\email{elena.campedel@univaq.it}

\thanks{E. Campedel was supported by GNSAGA-INDAM}

\author{Pedro Fagundes}
\address{Departamento de Matem\'atica, Universidade Federal de S\~ao Carlos, 13565-905, S\~ao Carlos, S\~ao Paulo, Brazil.}
\email{pedrofagundes@ufscar.br}

\thanks{P. Fagundes was supported by FAPESP grants 
2024/19129-8 and 2024/14914-9, and by CNPq grant 405779/2023-2. (Corresponding author)}

\author{Antonio Ioppolo}
\address{Dipartimento di Ingegneria e Scienze dell'Informazione e Matematica, Universit\`a degli Studi dell'Aquila, Via Vetoio 1, 67100, L'Aquila, Italy}
\email{antonio.ioppolo@univaq.it}

\thanks{A. Ioppolo was supported by “Progetti di Ateneo 2024” University of L'Aquila and by GNSAGA-INDAM}



\begin{abstract}
In this paper we consider the algebra of upper triangular matrices \( UT_n(F) \), endowed with a \(\mathbb{Z}_2\)-grading (superalgebra) and equipped with a superinvolution. These structures naturally arise in the context of Lie and Jordan superalgebras and play a central role in the theory of polynomial identities with involution, as showed in the framework developed by Aljadeff, Giambruno, and Karasik in \cite{AGK}. 

We provide a complete description of the identities of \( UT_4(F) \), where the grading is induced by the sequence \((0,1,0,1)\) and the superinvolution is the super-symplectic one. This work extends previous classifications obtained for the cases \( n = 2 \) and \( n = 3 \), and addresses an open problem for \( n \geq 4 \). 

In the last part of the paper, we investigate the image of multilinear polynomials on the superalgebra \( UT_n(F) \) with superinvolution, showing that the image is a vector space if and only if \( n \leq 3 \), thus contributing to an analogue of the L’vov-Kaplansky conjecture in this  context.
\end{abstract}

\maketitle

\section{\bf Introduction}
Let $F\langle x_1,x_2,\dots \rangle$ be the free associative algebra and let $A$ be an associative $F$-algebra. The image of a polynomial $f\in F\langle X \rangle$ on the algebra $A$ is naturally defined as the image of the polynomial function induced by $f$ on $A$ and we denote it by $f(A)$. The problem to study the algebraic structure of $f(A)$ has gained special attention in the recent years and it is related with important problems from Ring Theory. For instance, in case $f$ is a multilinear polynomial, one can asks if its image on some algebra inherits the linear structure of the later. This is in connection with the so called L'vov-Kaplansky conjecture which states that the image of a multilinear polynomial on the full matrix algebra is a vector space. A satisfactory answer to the L'vov-Kaplansky conjecture is only known in small cases (see \cite{albert,kanel-belov,shoda}). For the case of images of arbitrary polynomials on algebras we have the Waring type problems, which originally appeared in arithmetic but have spread to groups and algebras (see \cite{Matej} and references therein).

The case in which $A=UT_n(F)$, the algebra of upper triangular matrices, has also been investigated. In 2022, Gargate and de Mello (see \cite{gargate}) gave a  positive answer for an analogue of L'vov-Kaplansky conjecture where $F$ is an infinite field. See also \cite{fagundesplamen, fagundes} for the classification of images of multilinear polynomials on $UT_n(F)$ endowed with certain group gradings and (graded-)involutions.

The particular case where the image of a polynomial vanishes on some algebra is where the theory of polynomial identities takes its place. An important branch of research in PI-theory is to determine all polynomial identities satisfied by a given algebra $A$, or more precisely, to find a set of generators of the $T$-ideal of identities of $A$. For the case of $UT_n(F)$ this is already well known for an arbitrary field $F$. 

In the setting of algebras with additional structure, the graded identities of $UT_n(F)$ were described in \cite{DiVincenzoPlamenAngela} whereas the involution case was addressed in \cite{DiVincenzoKoshlukovLaScala2006} (just in the small cases $n=2$ and $n=3$). 

In this paper we deal with superalgebras (algebras graded by $\mathbb{Z}_2$) with superinvolution. They play a prominent role in the setting of Lie and Jordan superalgebras: in fact the skew-symmetric elements form a Lie superalgebra with respect to the super-commutator product, while the symmetric elements form a Jordan superalgebra with respect to the super-symmetrized product. More details can be found for instance in \cite{Kac1977, RacineZelmanov2003}. 

Within the theory of polynomial identities, the significance of superalgebras with superinvolution was highlighted in $2017$ by Aljadeff, Giambruno and Karasik. In \cite{AGK} they proved that any algebra with involution has the same identities as the Grassmann envelope of a finite-dimensional superalgebra with superinvolution. This result establishes a relationship similar to that demonstrated by Kemer between algebras and superalgebras, and it has led to a thorough study of this kind of algebras and their corresponding identities (see, for instance \cite{CentroneEstradaIoppolo2022, CostaIoppoloSantosVieira2021, Ioppolo2022, IoppoloMartino2021} and the references therein).

In this paper we focus our attention on the superalgebra $UT_n(F)$ of $n \times n $ upper triangular matrices endowed with a superinvolution, where $F$ is an algebraically closed field of characteristic zero. The superinvolutions on such a superalgebra were described in \cite{IoppoloMartino2018}. In the same paper, the authors gave a finite set of generators of the $*$-identities of $UT_2(F)$ and $UT_3(F)$ endowed with some superinvolution. However the description of the $*$-identities of $UT_n(F)$ is an open problem for $n\geq4$.

In this paper we shall give a complete description of the $*$-identities of $UT_4(F)$ endowed with the $\mathbb{Z}_2$-grading induced by the sequence $(0,1,0,1)$ and with the super-symplectic superinvolution. In the last section we also show that the images of multilinear $*$-polynomials which induce multilinear functions on $UT_n(F)$ is a vector space if and only if $n\leq 3$.

\section{\bf Superalgebras with superinvolution} \label{preliminaries}

Throughout the rest of the paper, unless otherwise stated, $F$ is a field of characteristic zero and all the algebras are assumed to be associative and to have the same ground field $F$. 
In this section we report  basic notions about \emph{superalgebras with superinvolution}.

An algebra $A$ is a  \emph{$\mathbb{Z}_2$-graded algebra} or  a \emph{superalgebra} if it has a vector space decomposition $A=A_0\oplus A_1$ such that $A_0 A_0 + A_1 A_1 \subseteq A_0$ and $A_0 A_1 + A_1 A_0 \subseteq A_1$. 
The elements of $A_{0}$ are called \emph{homogeneous of degree $0$} and those of $A_{1}$  \emph{homogeneous of degree $1$}. An element $a$ of $A$ is \emph{homogeneous} if it is homogeneous of degree $0$ or $1$ (and we denote its degree by $\vert a \vert_A$ or simply $\vert a \vert$ in absence of ambiguity). Analogously, a subalgebra or an ideal $I \subseteq A$ is \emph{homogeneous} or \emph{graded} if $I=(I\cap A_{0})\oplus  (I \cap A_{1})$.

A linear map $\ast \colon A \rightarrow A$ is said to be a \emph{superinvolution} of $A$ if 
\begin{itemize}
    \item[-] $*$ preserves the grading of the superalgebra, i.e., $A_i^* \subseteq A_i$, $i \in \{0, 1 \} $,

    \item[-] $*$ has order at most $2$,

    \item[-] for any homogeneous elements $a, b \in A$,
$$
(ab)^*=(-1)^{\vert a \vert \vert b \vert} b^* a^*.
$$ 
\end{itemize}

If $A_1^2=\{0\}$, then a superinvolution on $A$ is actually a graded involution. 

Here we are making use of the exponential notation: $a^* := *(a)$. 
In what follows, we shall freely use such a notation when it results more convenient.  

We refer to superalgebras with superinvolution simply as \emph{$\ast$-algebras}.

One could extend the notions of subalgebra and ideal also to this setting. 
In particular, recall that an ideal $I$ (subalgebra) of $A$ is a \emph{$\ast$-ideal} (\emph{subalgebra}) of $A$ if it is a graded ideal (subalgebra, respectively) such that $I^*=I$. The Jacobson radical $J = J(A)$ of a $\ast$-algebra $A$ is an example of a $\ast$-ideal. 

Last but not least,  recall that two $\ast$-algebras $(A, \ast)$ and $(A', \sharp)$ are \emph{isomorphic} (as $\ast$-algebras) if there exists an isomorphism of superalgebras  $\tau \colon A \rightarrow A'$ such that $\tau \left( a^* \right)=  \tau(a)^\sharp  $, for any $a\in A.$ 

\medskip

The rest of this section is devoted to recall the basic notions of the theory of polynomial identities in the setting of $\ast$-algebras.

Let  $F\gp{Y\cup Z}$ be the free algebra on the disjoint countable sets of variables $Y:=\{y_1, y_2, \ldots \}$ and $Z:=\{z_1, z_2, \ldots \}$. It has a natural superalgebra structure if we require that the variables from $Y$ have degree $0$ and those from $Z$ have degree $1$ (and then extend this grading to the monomials on $Y \cup Z$). This algebra  is said to be the \emph{free superalgebra} over $F$. 

We denote by $\Id^{sup}(A) = \{ f \in  F\gp{Y\cup Z}  \mid f \equiv 0 \mbox{ on } A \}$
the set of superpolynomial identities of $A$, which is a $T_2$-ideal of the free superalgebra, i.e., an ideal invariant under all graded endomorphisms of $F\gp{Y\cup Z}$.

Now, let $F\langle Y\cup Z, \ast \rangle$ be the free superalgebra with superinvolution (simply called the \emph{free $\ast$-algebra}) over $F$, where for each variable $y \in Y$, $y^*$ is homogeneous of  degree $0$, while for each variable $z\in Z$, $z^*$ is homogeneous of degree $1$.

Since $\mbox{char}\,F=0,$ given a $\ast$-algebra $A$, we can write 
\begin{equation} \label{symmetric and skew decomposition}
A=A_0^+ \oplus A_0^-\oplus A_1^+\oplus A_1^-,    
\end{equation} 
where, for any $i \in \{ 0, 1 \}$, $A_i^+=\{a\in A_i \ |
\ a^*=a\}$ and $A_i^-=\{a\in A_i \ | \ a^*=-a\}$ denote the sets of
symmetric and skew elements of $A_i$, respectively. This clarifies the reason why it is sometimes convenient to regard $F\gp{Y\cup Z, \ast}$ as generated by even symmetric and even skew variables, and by odd  symmetric and odd skew variables, namely
$$F\gp{Y\cup Z, \ast}=F\gp{y_1^+, y_1^-,z_1^+, z_1^-, y_2^+, y_2^-,z_2^+, z_2^-, \ldots },$$
where $y_i^+=y_i+y_i^\ast$, $y_i^-=y_i-y_i^\ast$, $z_i^+=z_i+z_i^\ast$ and $z_i^-=z_i-z_i^\ast$, $i\geq 1$.
We shall refer to its elements simply as \emph{polynomials}.

Now let $A$ be a $\ast$-algebra, as in (\ref{symmetric and skew decomposition}). A polynomial 
$$
f(y^+_1,\ldots,y^+_m, y^-_1,\ldots,y^-_n, z^+_1,\ldots,z^+_r, z^-_1,\ldots,z^-_s)\in F\langle Y\cup Z, \ast \rangle
$$
is a \emph{$\ast$-polynomial identity of $A$} (or simply a \emph{$\ast$-identity}) if, for all $u^+_1, \ldots, u^+_m\in A_0^+,$  $u^-_1, \ldots, u^-_n\in A^-_0,$ $v^+_1, \ldots, v^+_r\in A^+_1$ and $v^-_1, \ldots, v^-_s\in A^-_1$, one has that
$$
f \left(u^+_1, \ldots, u^+_m, u^-_1, \ldots, u^-_n, v^+_1, \ldots, v^+_r, v^-_1, \ldots, v^-_s \right)=0.
$$ 

We use the symbol $\Id^\ast(A)$ to indicate the set of all the $\ast$-identities satisfied by $A$,  which is a $T_2^\ast$-ideal of the free $\ast$-algebra, namely a two-sided ideal of the free $\ast$-algebra stable under every graded endomorphism of  $F\langle Y \cup Z \rangle$ commuting with the superinvolution $\ast$.

Since $F$ has characteristic zero, as in the ordinary case, it is easy to see that  $\Id^\ast(A)$ is completely determined by the multilinear polynomials it contains. We denote by
$$
P^{\ast}_n:=\mbox{span}_F\{w_{\sigma(1)} \cdots w_{\sigma(n)}\mid \sigma
\in S_n, \ w_i \in \{ y^+_i, y^-_i, z^+_i, z^-_i \},   \ i=1, \ldots, n\}
$$
the space of multilinear polynomials of degree $n$ in the variables
$y^+_1$, $y^-_1$, $z^+_1$, $z^-_1$, $\ldots$, $y^+_n$, $y^-_n$, $z^+_n$, $z^-_n$.

The study of $\Id^{\ast}(A)$ is equivalent to that of $P^{\ast}_n\cap \Id^{\ast}(A)$, for all $n\geq 1$. As made by Regev in the ordinary case (\cite{R}), one can define the $n$-th \emph{$\ast$-codimension} of the $\ast$-algebra $A$ as
$$
c^{\ast}_n(A):=\dim_F\frac{P^{\ast}_n}{P^{\ast}_n\cap \Id^{\ast}(A)},\ n\geq 1.
$$

Let $n\geq 1$ and write $n=n_1+\cdots+n_4$ as a sum of four non-negative integers. We denote by
$P_{n_1, \ldots, n_4}\subseteq P^{*}_n$
the vector space of the multilinear polynomials in which the first $n_1$ variables are even symmetric, the next $n_2$ variables are odd symmetric, the  next $n_3$ variables are even skew and the last $n_4$ variables are odd skew. Set 
$$c_{n_1, \ldots, n_4}(A)=\dim_F\dfrac{P_{n_1, \ldots, n_4}}{P_{n_1, \ldots, n_4}\cap \I^{*}(A)}.
$$
Hence it is immediate to see that
\begin{equation} \label{cod}
c_n^{*}(A)=\sum_{n_1+\cdots+n_4=n} \binom{n}{n_1, \ldots, n_4}c_{n_1, \ldots, n_4}(A),
\end{equation}
where $\binom{n}{n_1, \ldots, n_4}=\frac{n!}{n_1!\cdots n_4!}$ stands for the  multinomial coefficient.

If $A$ satisfies an ordinary polynomial identity, the sequence  $\{c_n^{\ast}(A)\}_{n\geq 1}$ is exponentially bounded (\cite{GiambrunoIoppoloLaMattina2016}). Moreover, in case $A$ is a finite-dimensional $*$-algebra, the analogue of the result of Giambruno and Zaicev holds (see \cite{Ioppolo2018}), namely that 
$$
\exp^\ast(A) := \lim_{n \to +\infty} \sqrt[n]{ c_n^{\ast}(A)}
$$
exists and it is a non-negative integer, called the \emph{$\ast$-exponent} of $A$.

\section{\bf Upper triangular matrices} \label{upper triangular}

In this section we focus our attention on the algebra $UT_n(F)$ of $n \times n$ upper triangular matrices over the field $F$, namely
\[
UT_n(F) = \left \{ \begin{pmatrix}
a_{11} & a_{12} & \cdots & a_{1n} \\
0 & a_{22} & \cdots & a_{2n} \\
\vdots &  & \ddots & \vdots \\
0 & \cdots & \cdots & a_{nn}\\
\end{pmatrix} \mid a_{11}, a_{12}, \ldots, a_{nn} \in F \right \}.
\]

Let $e_{kj}$ denote the usual matrix units, i.e., a matrix full of zeros except in the position $(k,j)$ in which a 1 appears. An $n$-tuple $(g_1, \ldots, g_{n})\in \mathbb{Z}_2^{n}$ defines an elementary $\mathbb{Z}_2$-grading on $UT_{n}(F)$ by setting
\begin{equation*}
\begin{split}
\left(UT_n(F) \right)_0 &= \mbox{span} \left \{ e_{kj} \mid g_k + g_j \equiv 0 \ (\mbox{mod} \ 2) \right \}, \\
\left(UT_n(F) \right)_1 &= \mbox{span} \left \{ e_{kj} \mid g_k + g_j \equiv 1 \ (\mbox{mod} \ 2) \right \}.    
\end{split}
\end{equation*}

The importance of such gradings is highlighted in the following result (\cite{DiVincenzoKoshlukovValenti2004, ValentiZaicev2003, ValentiZaicev2007}).

\begin{Theorem}   \label{the G gradings on UTn are elementary}
Any $\mathbb{Z}_2$-grading on $UT_n(F)$ is equivalent to an elementary one. Moreover, there exist $2^{n-1}$ non-isomorphic $\mathbb{Z}_2$-gradings on $UT_n(F)$.
\end{Theorem}

Our next goal is to present the classification of the superinvolutions that can be defined on the algebra $UT_n(F)$. 
We start defining two involutions on $UT_n(F)$.

\begin{enumerate}
    \item[1.] The reflection involution $\circ \colon UT_n(F) \rightarrow UT_n(F)$ is given by the formula
    \[
    e_{ij}^\circ=e_{n+1-j,n+1-i}
    \]
    where $1\leq i\leq j\leq n$.
    \vspace{0.1 cm}

    \item[2.] The symplectic involution $s \colon UT_n(F)\rightarrow UT_n(F)$ can be defined just in case $n=2m$ by $Y^s=JY^\circ J$, where
    \[
    J=\begin{pmatrix}
        I_m&0\\
        0&-I_m
    \end{pmatrix}
    \] and $I_m$ is the identity matrix of order $m$.
\end{enumerate}

If the $\mathbb{Z}_2$-grading on $UT_n(F)$ is induced by a sequence $(g_1,\dots,g_n)$ such that 
\[
g_1+g_n=g_2+g_{n-1}=\cdots=g_n+g_1
\]
it follows that $\circ$ and $s$ are actually graded involutions on $UT_n(F)$.

\medskip

Let us now recall the definition of a specific superautomorphism that applies to any superalgebra with nilpotent odd part, following the general framework introduced in \cite{Ioppolo2022}. For the particular case of the algebra $UT_n(F)$ an older definition can be found in \cite[Definition 3.2]{IoppoloMartino2018}.

\begin{Definition}
The superautomorphism $\phi \colon UT_n(F) \rightarrow UT_n(F)$ is the map given on homogeneous elements $a_0 \in A_0$ and $a_1 \in A_1$ by putting
\begin{equation*}
\begin{split}
\phi(a_0) &= \begin{cases} a_0 & \mbox{ if } a_0 \in \left(UT_n(F)_1\right)^{2k} \mbox{ but } a_0 \not \in \left(UT_n(F)_1\right)^{2k+2} \mbox{ and } k \mbox{ is even} \\
-a_0 & \mbox{ otherwise } \end{cases}, \\
\phi(a_1) &= \begin{cases} a_1 & \mbox{ if } a_1 \in \left(UT_n(F)_1\right)^{2k+1} \mbox{ but } a_1 \not \in \left(UT_n(F)_1\right)^{2k+3} \mbox{ and } k \mbox{ is even} \\
-a_1 & \mbox{ otherwise } \end{cases}.
\end{split}
\end{equation*}
\end{Definition}

Composing the two involutions above with the superautomorphism $\phi$ we get the following two superinvolutions on $UT_n(F)$.

\begin{enumerate}
    \item[1.] The super-reflection superinvolution is the map $\bar{\circ} = \circ\phi$.
    \vspace{0.1 cm}
    
    \item[2.] The super-symplectic superinvolution is the map $\bar{s} = s \phi$.
\end{enumerate}

Now we are in a position to present the classification of the superinvolutions on $UT_n(F)$ (see \cite{IoppoloMartino2018}).

\begin{Theorem} \label{superinvolutions on UTn}
Let $UT_n(F)$ be endowed with an elementary $\mathbb{Z}_2$-grading defined by the $n$-tuple $(g_1,\dots,g_n)$. Then there exists a superinvolution $*\colon UT_n(F)\rightarrow UT_n(F)$ if and only if 
    \[
    g_1+g_n=g_2+g_{n-1}=\cdots =g_n+g_1.
    \]
Moreover, every superinvolution on $UT_n(F)$ with the elementary grading above is equivalent to either $\bar{\circ}$ or $\bar{s}$.
\end{Theorem}

For simplicity, we will refer to the elementary gradings induced by the above kind of sequences as $*$-type.

\section{\bf On the identities of upper triangular matrices} \label{identities upper triangular}

An important problem in PI-theory is determining a basis of the polynomial identities for a given algebra. In the case of upper triangular matrices, a general result was given by Malcev in \cite{Malcev1971} (ordinary identities, no additional structure). In the setting of superalgebras with superinvolution, analogous results have been achieved when the matrix size is relatively small. In \cite{IoppoloMartino2018}, the authors fully addressed the cases $n=2$ and $n=3$, for any $\mathbb{Z}_2$-grading and any superinvolution. It is important to note, however, that when dealing with the trivial grading, the results presented in the aforementioned paper are merely a reformulation of the theorems from \cite{DiVincenzoKoshlukovLaScala2006}, which was carried out for algebras with involution.

\medskip

Since we will also use them in the last section of the paper, here we report the results concerning the algebra $UT_3(F)$ with the elementary grading given by $(0,1,0)$ and the only superinvolution that can be defined in this case, namely the super-reflection $\bar{\circ}$ given by:
\[
\begin{pmatrix}
    a&b&c\\
    0&d&e\\
    0&0&f
\end{pmatrix}^{\bar{\circ}}=\begin{pmatrix}
    f&e&-c\\
    0&d&b\\
    0&0&a
\end{pmatrix}.
\]

\begin{Theorem}\label{identities of ut3}
The $T_2^*$-ideal of identities of the $*$-algebra $UT_3(F)$ with elementary $\mathbb{Z}_2$-grading defined by the triple $(0,1,0)$ and superinvolution $\bar{\circ}$ is generated, as a $T_2^*$-ideal, by the following polynomials:
\begin{equation*}
\begin{split}
&1. \ [y_1^+, y_2^+], \ \ \ \ 2. \ [y^+, y^-], \ \ \ \ 3. \ y_1^- y_2^- y_3^- - y_3^- y_2^- y_1^-, \ \ \ \ 4. \ [y_1^-, y_2^-][y_3^-, y_4^-],  \\
&5. \ z[y_1^-, y_2^-], \ \ \ 6. \ [y_1^-, y_2^-]z, \ \ \ 7. \ y_1^- z y_2^-, \ \ \ 8. \ z_1 y^- z_2, \ \ \ 9. \ z_1 y^+ z_2 - z_2 y^+ z_1, \\
&10. \ [z_1^+, z_2^+], \ \ \ 11. \ [z_1^-, z_2^-], \ \ \ 12. \ z^+ z^- + z^- z^+, \ \ \ \ \ \ \ \ \ 13. \  z_1 z_2 z_3.
\end{split}
\end{equation*}
\end{Theorem}

The main goal of this section is to deal with the algebra $UT_4(F)$ of $4 \times 4$ upper triangular matrices. Apart from the trivial grading, on this algebra we can define the following elementary gradings of $*$-type: $(0,0,1,1)$, $(0,1,1,0)$ and $(0,1,0,1)$. According to Theorem \ref{superinvolutions on UTn}, on all the above superalgebras we can define two, non-equivalent, superinvolutions. 

\medskip

Let us consider the superalgebra $A = UT_4(F)$ with elementary grading induced by $(0,1,0,1)$ and endowed with the super-symplectic superinvolution
\[
\begin{pmatrix}
a & b & c & d \\
0 & e & f & g \\
0 & 0 & h & i \\
0 & 0 & 0 & l
\end{pmatrix}^{\bar{s}} = 
\begin{pmatrix}
l & i & g & d \\
0 & h & -f & c \\
0 & 0 & e & b \\
0 & 0 & 0 & a
\end{pmatrix}.
\]

Clearly, the symmetric and skew homogeneous subspaces of $A$ are the following:
\begin{equation*}
\begin{split}
A_0^{+} &= \mbox{span}_F \{e_{11}+e_{44}, e_{22}+e_{33}, e_{13}+e_{24} \}; \\
A_0^{-} &= \mbox{span}_F \{e_{11}-e_{44}, e_{22}-e_{33}, e_{13}-e_{24} \}; \\
A_1^{+} &= \mbox{span}_F \{e_{14}, e_{12}+e_{34} \}; \\
A_1^{-} &= \mbox{span}_F \{e_{23}, e_{12}-e_{34} \}.
\end{split}
\end{equation*}

In the following lemmas we collect some $*$-identities satisfied by the algebra $A$. In order to simplify the notation we denote by $y_i$ any variable of homogeneous degree $0$, by $z_i$ any variable of homogeneous degree $1$ and by $x_i$ any variable. 

\begin{Lemma}\label{identities of ut4}
The following are $*$-identities of $A$:
\begin{enumerate}


\item\label{id:2} $[y_1, y_2]x[y_3,y_4] \equiv 0$;
\vspace{0.1 cm}

\item\label{id:3} $[x_1, x_2, x_3] - [x_3, x_2, x_1]  - [x_1, x_3, x_2]\equiv 0$ \  \mbox{(Jacobi identity)};
\vspace{0.1 cm} 




\item\label{id:7} $[y_1, y_2, y_{i_3}, \ldots, y_{i_k}] - [y_1, y_2, y_{3}, \ldots, y_{k}] \equiv 0$, $k\geq 2$;
\vspace{0.1 cm}

















\item \label{id:24} $z_1^+ y z_2^+ \equiv 0$; 
\vspace{0.1 cm}

\item\label{id:25} $z_1^{+} z_2 z_3^{+} - z_3^{+} z_2 z_1^{+}  \equiv 0$;
\vspace{0.1 cm}

\item\label{id:26} $z_1^{-} z_2 z_3^{-} - z_3^{-} z_2 z_1^{-}  \equiv 0$;
\vspace{0.1 cm}


\item\label{id:28} $z_1 z_2 z_3 z_4 \equiv 0$;
\vspace{0.1 cm}

\item\label{id:32} $z_1^- y_4 z_2^+ y_5 z_3^- \equiv 0$;
\vspace{0.1 cm}

\item\label{id:33} $z_1^- y_2 z_3^- y_4 z_5^+ + z_5^+ y_2 z_3^- y_4 z_1^- \equiv 0$;
\vspace{0.1 cm}

\item \label{id:34} $ [y_1, \ldots, y_a] z y^{+} - y^{+} [y_1,\ldots,y_a] z \equiv 0$;
\vspace{0.1 cm}

\item \label{id:35} $  [y_1, \ldots, y_a] z y^{-} + y^{-} [y_1,\ldots,y_a] z \equiv 0$.




\end{enumerate}
\end{Lemma}
\begin{proof}
The result is easily obtained by performing simple calculations.
\end{proof}

Now let  $W = \{ w_1, \ldots, w_k \}$ be a set of symmetric and skew even variables, that is $w_i = y_i^+$ or $w_i = y_i^-$. Denote by $g$ the number of skew even variables inside $W$.

\begin{Lemma}\label{identities hard of ut4}
The following are $*$-identities of $A$:
\begin{itemize}

\item[(i)]  $(w_{1} \cdots w_{k} -  w_{\sigma(1)} \cdots w_{\sigma(k)}) z^+
+(-1)^{g+1} z^+ (w_{k} \cdots w_{1} - w_{\sigma(k)} \cdots w_{\sigma(1)}) \equiv 0 $, 

\item[(ii)]  $(w_{1} \cdots w_{k} -  w_{\sigma(1)} \cdots w_{\sigma(k)}) z^-
+(-1)^{g} z^- (w_{k} \cdots w_{1} - w_{\sigma(k)} \cdots w_{\sigma(1)}) \equiv 0$. 

\end{itemize}
\end{Lemma}

\begin{proof}
We prove the first item, as the second can be derived in an analogous manner.

Let us consider the following evaluation: 
\begin{equation*}
\begin{split}
w_i &\leadsto \alpha_i (e_{11} \pm e_{44}) + \beta_i (e_{22} \pm e_{33}) + \gamma_i (e_{13} \pm e_{24}), \ i=1, \ldots, k;  \\
z^+ &\leadsto a e_{14} + b (e_{12} + e_{34}).
\end{split}
\end{equation*}


We claim that $w_{1}\cdots w_{k}$ evaluates to
$$ \alpha_{1}\cdots\alpha_{k}\left(e_{11}+(-1)^g e_{44}\right)
+ \beta_{1}\cdots\beta_{k}\left(e_{22}+(-1)^g e_{33}\right) + \mathcal{S}^{1,\ldots,k} e_{13}
+ \mathcal{R}^{1,\ldots,k} e_{24},
$$
where 
\begin{equation*}
\begin{split}
\mathcal{S}^{1,\dots,k} &= \sum_{m=1}^{k} \alpha_{1}\cdots\alpha_{{m-1}} \gamma_{m} \beta_{{m+1}}\cdots\beta_{k}  \ \sgn(w_{{m+1}})\cdots \sgn(w_{{k}}), \\
\mathcal{R}^{1,\ldots,k} &= \sgn(w_{{m}}) \sum_{m=1}^{k}  \beta_{1}\cdots\beta_{{m-1}} \gamma_{m} \alpha_{{m+1}}\cdots\alpha_{k}  \ \sgn(w_{{m+1}})\cdots \sgn(w_{{k}}),
\end{split}
\end{equation*}
where $\sgn(w)=1$ if $w=y^+$ and $\sgn(w)=-1$ otherwise.
Notice that $\alpha_{j}=1$ and $\beta_{j} = 1$ if $j<0$ or $j>k$.

\smallskip

To show this, one can easily check that
$$
\mathcal{S}^{1, 2} = \alpha_{1}\gamma_{2} + \sgn(w_{2}) \gamma_{1} \beta_{2}. 
$$
Moreover, using an easy induction argument, when one multiplies the evaluation of $w_{1} \cdots w_{{k-1}}$ by the evaluation of $w_{k}$, the coefficient of $e_{13}$ is precisely 
$$
\mathcal{S}^{1, \dots, {k-1}, k} =  \alpha_{1} \cdots \alpha_{{k-1}}\gamma_{k} + \sgn(w_{k}) \mathcal{S}^{1, \dots, {k-1}}\cdot \beta_{k}.
$$

Finally, looking at the coefficients of $e_{24}$, we get that 
\begin{equation*}
\begin{split}
\mathcal{R}^{1, 2} &= \sgn(w_{2}) \left(\beta_{1} \gamma_{2}  + \sgn(w_{1}) \gamma_{1}\alpha_{2}\right), \\
\mathcal{R}^{1, \dots, {k-1}, k} &= \sgn(w_{k}) \left( \beta_{1} \cdots \beta_{{k-1}}\gamma_{k} +  \mathcal{R}^{1, \dots, {k-1}}\cdot \alpha_{k} \right).  
\end{split}
\end{equation*}

In what follows, we make use of the following equality:
\begin{equation} 
	\label{eq:s1-s2}
\mathcal{S}^{1, \dots, k} = (-1)^g \cdot \mathcal{R}^{k, \ldots, 1}.
\end{equation}

Now we are ready to finalize the proof. In fact, $w_{1}\cdots w_{k} z^{+}$ evaluates to
$$ 
\alpha_{1}\cdots\alpha_{k}b e_{12} +
(-1)^g \beta_{1}\cdots\beta_{k} b e_{34} +
\left( \alpha_{1}\cdots\alpha_{k} a + \mathcal{S}^{1,\dots,k} b \right) e_{14}.
$$

Similarly, according to~\eqref{eq:s1-s2}, it follows that $z^{+} w_{k}\cdots w_{1}$ evaluates to
\[
\beta_{1}\cdots\beta_{k}b e_{12} +
(-1)^g \alpha_{1} \cdots \alpha_{k} b e_{34} +
(-1)^g \left(\alpha_{1}\cdots\alpha_{k} a + \mathcal{S}^{1,\dots,k} b \right) e_{14}.
\]

Hence, the evaluation of $w_{1}\cdots w_{k} z^{+} + (-1)^{g+1} z^{+} w_{k}\cdots w_{1}$ is
$$ 
\left(\alpha_{1}\cdots\alpha_{k} - (-1)^{g} \beta_{1}\cdots\beta_{k} \right) b (e_{12} - e_{34}).
$$

Note that if we change the order of the indices, the evaluation does not change: in fact, for any $\sigma \in S_k$, $w_{\sigma(1)} \cdots w_{\sigma(k)} z^{+} + (-1)^{g+1} z^{+} w_{\sigma(k)}\cdots w_{\sigma(1)}$ evaluates to
$$ 
\left(\alpha_{\sigma(1)}\cdots\alpha_{\sigma(k)} - (-1)^{g} \beta_{\sigma(1)}\cdots\beta_{\sigma(k)}\right) b (e_{12} - e_{34}),
$$
and this concludes the proof.
\end{proof}

In order to present the next result, let us consider the following polynomials: 
\begin{equation*}
\begin{split}
P_1 &=  y_{i_1}^{+}\cdots y_{i_a}^{+} y_{i_{a+1}}^{-}\cdots y_{i_{a+a'}}^{-} z_1^{-} y_{j_1}^{+} \cdots y_{j_b}^{+} y_{j_{b+1}}^{-} \cdots y_{j_{b+b'}}^{-} z_2^{-} y_{l_1}^{+} \cdots y_{l_c}^{+}y_{l_{c+1}}^{-} \cdots y_{l_{c+c'}}^{-}, \\
P_2 &= z_1^{-} y_{i_1}^{+}\cdots y_{i_a}^{+}y_{i_{a+1}}^{-}\cdots y_{i_{a+a'}}^{-} y_{j_1}^{+} \cdots y_{j_b}^{+} y_{j_{b+1}}^{-} \cdots y_{j_{b+b'}}^{-} z_2^{-} y_{l_1}^{+} \cdots y_{l_c}^{+}y_{l_{c+1}}^{-} \cdots y_{l_{c+c'}}^{-}, \\
P_3 &= z_1^{-} y_{i_1}^{+}\cdots y_{i_a}^{+}y_{i_{a+1}}^{-}\cdots y_{i_{a+a'}}^{-} y_{j_1}^{+} \cdots y_{j_b}^{+} y_{j_{b+1}}^{-} \cdots y_{j_{b+b'}}^{-} y_{l_1}^{+} \cdots y_{l_c}^{+}y_{l_{c+1}}^{-} \cdots y_{l_{c+c'}}^{-} z_2^{-}, \\
P_4 &= y_{i_1}^{+}\cdots y_{i_a}^{+}y_{i_{a+1}}^{-}\cdots y_{i_{a+a'}}^{-} z_1^{-} y_{j_1}^{+} \cdots y_{j_b}^{+} y_{j_{b+1}}^{-} \cdots y_{j_{b+b'}}^{-} y_{l_1}^{+} \cdots y_{l_c}^{+}y_{l_{c+1}}^{-} \cdots y_{l_{c+c'}}^{-} z_2^{-}.
\end{split}
\end{equation*}

\begin{Lemma}\label{identities ut4 n3 + n4 = 2}
The following are $*$-identities of $A$:
\begin{itemize}

\item[(a)] \label{id-(n_1,n_2,0,2)-1} 
$w_{\sigma(1)} \cdots w_{\sigma(k)} z_{1} z_{2} - w_{1} \cdots w_{k} z_{1} z_{2} \equiv 0$;
\vspace{0.1 cm}

\item[(b)] \label{id-(n_1,n_2,0,2)-2} 
$z_{1} w_{\sigma(1)} \cdots w_{\sigma(k)} z_{2} - z_{1} w_{1} \cdots w_{k} z_{2} \equiv 0$;
\vspace{0.1 cm}

\item[(c)] \label{id-(n_1,n_2,0,2)-3} 
$z_{1} z_{2} w_{\sigma(1)} \cdots w_{\sigma(k)} -  z_{1} z_{2} w_{1} \cdots w_{k} \equiv 0$;
\vspace{0.1 cm}

\item[(d)] \label{id-(n_1,n_2,0,2)-4} 
$z^{-} w_{1} \cdots w_{k} z^{+} + (-1)^{g+1} w_{1} \cdots w_{k} z^{-} z^{+} \equiv 0$;
\vspace{0.1 cm}

\item[(e)] \label{id-(n_1,n_2,0,2)-5} 
$z^{+} w_{1} \cdots w_{k} z^{-} + (-1)^{g+1} z^{+} z^{-} w_{1} \cdots w_{k} \equiv 0$;
\vspace{0.1 cm}

\item[(f)] \label{id-(n_1,n_2,0,2)-6} 
$(-1)^{a'+c'} P_1 - (-1)^{c'} P_2 + P_3 - (-1)^{a'} P_4 \equiv 0$.

\end{itemize}

\end{Lemma}
\begin{proof}
The proof consists of straightforward but lengthy and tedious computations, and is therefore omitted.
\end{proof}

\begin{Lemma}\label{identities ut4 n3 + n4 = 3}
The following are $*$-identities of $A$:
\begin{itemize}

\item[(a)] \label{id-(n_1,n_2,0,3)-2} 
$z_1 z_2 w_{1} \cdots w_{k} z_3 - (-1)^{g} z_1 w_{1} \cdots w_{k} z_2 z_3 \equiv 0$;
\vspace{0.1 cm}



\item[(b)] \label{id-(n_1,n_2,0,3)-4} 
$(-1)^gw_1\cdots w_kz_1z_2z_3-z_1z_2z_3w_1\cdots w_k\equiv0$.

\end{itemize}

\end{Lemma}
\begin{proof}
The proof reduces to a straightforward computation, taking into account that the evaluation of any monomial involving three odd variables always lies in $\mbox{span}_F \{ e_{14} \}$.
\end{proof}

Now let us prove the main result of this section. 

\begin{Theorem}\label{ideal of ut4}
Let $A$ be the $*$-algebra $UT_4(F)$ with elementary $\mathbb{Z}_2$-grading induced by $(0,1,0,1)$ and superinvolution  $\bar{s}$. Then the $T_2^*$-ideal of identities of this algebra is generated by the polynomials listed in Lemmas \ref{identities of ut4}, \ref{identities hard of ut4}, \ref{identities ut4 n3 + n4 = 2} and \ref{identities ut4 n3 + n4 = 3}.
\end{Theorem}
\begin{proof}
Let $J$ be the $T_2^*$-ideal generated by the polynomials listed in Lemmas \ref{identities of ut4}, \ref{identities hard of ut4}, \ref{identities ut4 n3 + n4 = 2} and \ref{identities ut4 n3 + n4 = 3}. Clearly we have that $J \subseteq \I^*(A)$.

In order to prove the opposite inclusion, first we find a set of generators of $P_n^*$ modulo $P_n^*\cap J$, for all $n\geq 1$. To this end, let us consider $P^*_{n_1,\ldots , n_4}$, where $n_1+\cdots +n_4=n.$ Since $z_1 z_2 z_3 z_4 \equiv 0$, then it must be $n_3+n_4 \leq 3.$ Thus we have to consider the following cases. 

\medskip

\noindent \textbf{Case 1:} $n_3 = n_4 = 0$. 

\smallskip

In this case we are dealing with polynomials containing just even variables. By the Poincar\'{e}-Birkhoff-Witt theorem, any polynomial can be written as a linear combination of products of the type 
$$
y^+_{i_1}\cdots y^+_{i_a}y^-_{j_1}\cdots y^-_{j_b} \omega_1 \cdots \omega_t,
$$
where $\omega_1,\ldots, \omega_t$ are left normed commutators in the $y^+_i$'s and $y^-_i$'s, $i_1< \cdots <i_a$ and $j_1 <\cdots <j_b$.
By the identities~\eqref{id:2}, ~\eqref{id:3} and~\eqref{id:7}, we can partially order the variables inside the commutator and we can distinguish four types of generators: 
\begin{enumerate}[(A)]

\item \label{n1,n2,0,0:gen:A} $y_{i_1}^{+} \dots y_{i_a}^{+}y_{j_1}^{-}\dots y_{j_b}^{-}$,
\vspace{0.1 cm}

\item \label{n1,n2,0,0:gen:B} $y_{i_1}^{+} \dots y_{i_a}^{+}y_{j_1}^{-}\dots y_{j_b}^{-}[y_{k}^{+}, y_{l_1}^{+}, \dots, y_{l_c}^{+}, y_{m_1}^{-}, \dots, y_{m_d}^{-}]$, where 
    $$
    k>l_1<\dots < l_c, \ m_1<\dots <m_d, 
    $$

\item \label{n1,n2,0,0:gen:C} $y_{i_1}^{+} \dots y_{i_a}^{+}y_{j_1}^{-}\dots y_{j_b}^{-}[y_{k}^{-}, y_{l_1}^{-}, \dots, y_{l_c}^{-}, y_{m_1}^{+}, \dots, y_{m_d}^{+}]$, where 
    $$ k>l_1<\dots < l_c, \ m_1<\dots <m_d, 
    $$

\item\label{n1,n2,0,0:gen:D} $y_{i_1}^{+} \dots y_{i_a}^{+}y_{j_1}^{-}\dots y_{j_b}^{-}[y_{l_1}^{+}, y_{m_1}^{-}, \dots, y_{m_d}^{-}, y_{l_2}^{+}, \dots, y_{l_c}^{+}]$, where 
    $$ 
    m_1<\dots <m_d, \ l_1< l_2<\dots < l_c.
    $$      
\end{enumerate}

Notice that in the last three types, the commutators appear and the first two variables are exactly of the required kind. 

We next show that these polynomials are linearly independent modulo $\I^*(A)$. To this end, let $f\in \I^*(A)$ be a linear combination of the above polynomials and fix a type (A) polynomial $y_{i_1}^+\cdots y_{i_a}^+y_{j_1}^-\cdots y_{j_b}^-$. Performing the evaluation 
$$
\mathcal{E}_1:
\begin{cases}
	y_{i_1}^{+}, \dots, y_{i_a}^{+} \leadsto I \\
	y_{j_1}^{-}, \dots, y_{j_b}^{-} \leadsto e_{11}-e_{44} 
\end{cases}
$$
we see that the polynomial above evaluates to $e_{11}+(-1)^be_{44}$ while all other polynomials of types \eqref{n1,n2,0,0:gen:B}, \eqref{n1,n2,0,0:gen:C} and \eqref{n1,n2,0,0:gen:D} vanish. Therefore the coefficient of the polynomial of type \eqref{n1,n2,0,0:gen:A} is zero.

Let us now consider the type \eqref{n1,n2,0,0:gen:B} polynomial
\[
y_{i_1}^{+} \dots y_{i_a}^{+}y_{j_1}^{-}\dots y_{j_b}^{-}[y_{k}^{+}, y_{l_1}^{+}, \dots, y_{l_c}^{+}, y_{m_1}^{-}, \dots, y_{m_d}^{-}] 
\]
where $a$ is maximal, and $b$ is maximal among all the polynomials of type (A) with $a$ maximal. Consider the evaluation
$$
\mathcal{E}_2:
\begin{cases}
	y_{i_1}^{+}, \dots, y_{i_a}^{+} \leadsto I \\
	y_{j_1}^{-}, \dots, y_{j_b}^{-} \leadsto e_{11} - e_{22} + e_{33} -e_{44} \\
	y_k^{+} \leadsto e_{13}+e_{24} \\ 
	y_{l_1}^{+}, \dots, y_{l_c}^{+} \leadsto e_{11}+e_{44} \\
	y_{m_1}^{-}, \dots, y_{m_d}^{-} \leadsto e_{22}-e_{33}
\end{cases}.
$$
It follows that $y_{i_1}^{+} \dots y_{i_a}^{+}y_{j_1}^{-}\dots y_{j_b}^{-}[y_{k}^{+}, y_{l_1}^{+}, \dots, y_{l_c}^{+}, y_{m_1}^{-}, \dots, y_{m_d}^{-}] $ evaluates to 
\[
(-1)^d\left( (-1)^c e_{13} +(-1)^b e_{24} \right), 
\]
while all the other polynomials evaluate to $0$
(note that $y_{l_1}^{+}$ in~\eqref{n1,n2,0,0:gen:D} can never be $y_{k}^{+}$ in~\eqref{n1,n2,0,0:gen:B}, and also that the matrix $e_{11}-e_{22}+e_{33}-e_{44}$ lies in the centralizer of the set $\{e_{13},e_{24}\}$). Therefore the coefficient of the polynomial above is zero. By varying the indices, we obtain that the coefficients of all the polynomials of type~\eqref{n1,n2,0,0:gen:B} with $a$ and $b$ maximal are zero as well. We proceed recursively to get the same conclusion for all the remaining polynomials of type~\eqref{n1,n2,0,0:gen:B}.  

Now we consider a type \eqref{n1,n2,0,0:gen:D} 
\[
y_{i_1}^{+} \dots y_{i_a}^{+}y_{j_1}^{-}\dots y_{j_b}^{-}[y_{l_1}^{+}, y_{m_1}^{-}, \dots, y_{m_d}^{-}, y_{l_2}^{+}, \dots, y_{l_c}^{+}],
\]
where $a$ and $b$ are maximal as before. Under the evaluation
$$
\mathcal{E}_3:
\begin{cases}
	y_{i_1}^{+}, \dots, y_{i_a}^{+} \leadsto I \\
	y_{j_1}^{-}, \dots, y_{j_b}^{-} \leadsto e_{11} - e_{22} + e_{33} -e_{44} \\
	y_{l_1}^{+} \leadsto e_{13}+e_{24} \\ 
	y_{l_2}^{+}, \dots, y_{l_c}^{+} \leadsto e_{22}+e_{33} \\
	y_{m_1}^{-}, \dots, y_{m_d}^{-} \leadsto e_{11}-e_{44}
\end{cases},
$$
we get that the polynomial above evaluates to $(-1)^d(e_{13}+(-1)^{b+c}e_{24})$ while all the others ones evaluates to $0$. Hence the coefficient of this polynomial is zero. By varying the coefficients and using recursion as we drop the maximal length of $y^+$ and $y^-$ on the left of the commutator, we get that the coefficients of all type \eqref{n1,n2,0,0:gen:D} polynomials are zero.


Finally, in order to show that the coefficients of all the polynomials of type~\eqref{n1,n2,0,0:gen:C} are zero it is sufficient to start with the polynomial
\[
y_{i_1}^{+} \dots y_{i_a}^{+}y_{j_1}^{-}\dots y_{j_b}^{-}[y_{k}^{-}, y_{l_1}^{-}, \dots, y_{l_c}^{-}, y_{m_1}^{+}, \dots, y_{m_d}^{+}]
\]
with $a$ and $b$ maximal and perform the evaluation
$$
\mathcal{E}_4:
\begin{cases}
	y_{i_1}^{+}, \dots, y_{i_a}^{+} \leadsto I \\
	y_{j_1}^{-}, \dots, y_{j_b}^{-} \leadsto e_{11} - e_{22} + e_{33} -e_{44} \\
	y_k^{-} \leadsto e_{13}-e_{24} \\ 
	y_{l_1}^{-}, \dots, y_{l_c}^{-} \leadsto e_{11}-e_{44} \\
	y_{m_1}^{+}, \dots, y_{m_d}^{+} \leadsto e_{22}+e_{33}
\end{cases}.
$$

In fact, the polynomial above is the only one with a non-zero evaluation, namely
\[ 
(-1)^c \left( e_{13} + (-1)^{b+d+1} e_{24} \right).
\]

\medskip

\noindent \textbf{Case 2:} $n_3+ n_4 = 1$. 

\smallskip

Assume first $n_3 = 1 $ and $ n_4 = 0$. Let $n = n_1 + n_2 + 1$ and write $W = \{ w_1, \ldots, w_{n_1 + n_2} \}$ to denote the set of symmetric and skew even variables, that is $w_i = y_i^+$ or $w_i = y_i^-$. In this case, we are dealing with polynomials of the kind
\begin{equation}  \label{polys n1 n2 1}
w_{\sigma(1)} \cdots w_{\sigma(i)} z^+ w_{\sigma(i+1)} \cdots w_{\sigma(n_1 + n_2)}, \ \sigma \in S_{n_1+n_2}.
\end{equation}

By identity $(i)$ of Lemma \ref{identities hard of ut4}, we can order the ``right side'' in~\eqref{polys n1 n2 1}:
$$
w_{\sigma(1)} \cdots w_{\sigma(i)} z^+ 
y_{r_1}^{+} \cdots y_{r_e}^{+} y_{s_1}^{-} \cdots y_{s_f}^{-},
$$
where $r_1< \cdots <r_e$, $s_1< \cdots <s_f$.

By Poincar\'e-Birkhoff-Witt theorem and identity~\eqref{id:2} we can order the ``left side'':
\begin{equation} \label{eq:polys-2}
y_{i_1}^{+} \cdots y_{i_a}^{+} y_{j_1}^{-} \cdots y_{j_b}^{-} [w_1, \ldots, w_c] z^{+} y_{r_1}^{+}\cdots y_{r_e}^{+} y_{s_1}^{-} \cdots y_{s_f}^{-},
\end{equation}
where $i_1< \cdots <i_a$, $j_1< \cdots <j_b$, and the variables inside the commutator are ordered as in~\eqref{n1,n2,0,0:gen:A}, \eqref{n1,n2,0,0:gen:B} or~\eqref{n1,n2,0,0:gen:C} of Case 1.

When the commutator actually appears, according to the identities~\eqref{id:34} and~\eqref{id:35}, we may assume that $z^{+}$ is in the last position.

In conclusion we have the following types of generators:
\begin{enumerate}[(A)]

\item\label{n1,n2,1,0:gen:D} $y_{i_1}^{+} \dots y_{i_a}^{+}y_{j_1}^{-}\dots y_{j_b}^{-} z^+ y_{r_1}^{+} \cdots y_{r_e}^{+} y_{s_1}^{-} \cdots y_{s_f}^{-}$, where 
    $$ 
    i_1< \cdots <i_a, \ j_1 < \cdots <j_b, \ r_1< \cdots < r_e, \ s_1 < \cdots <s_f,
    $$

	\item \label{n1,n2,1,0:gen:A} $y_{i_1}^{+} \dots y_{i_a}^{+}y_{j_1}^{-}\dots y_{j_b}^{-}[y_{k}^{+}, y_{l_1}^{+}, \dots, y_{l_c}^{+}, y_{m_1}^{-}, \dots, y_{m_d}^{-}] z^+$, where 
    $$
   i_1< \cdots <i_a, \ j_1 < \cdots <j_b, \ k>l_1<\dots < l_c, \ m_1<\dots <m_d, 
    $$
	\item \label{n1,n2,1,0:gen:B} $y_{i_1}^{+} \dots y_{i_a}^{+}y_{j_1}^{-}\dots y_{j_b}^{-}[y_{k}^{-}, y_{l_1}^{-}, \dots, y_{l_c}^{-}, y_{m_1}^{+}, \dots, y_{m_d}^{+}] z^+$, where 
    $$i_1< \cdots <i_a, \ j_1 < \cdots <j_b, \ k>l_1<\dots < l_c, \ m_1<\dots <m_d, 
    $$
	\item\label{n1,n2,1,0:gen:C} $y_{i_1}^{+} \dots y_{i_a}^{+}y_{j_1}^{-}\dots y_{j_b}^{-}[y_{l_1}^{+}, y_{m_1}^{-}, \dots, y_{m_d}^{-}, y_{l_2}^{+}, \dots, y_{l_c}^{+}] z^+$, where 
    $$ 
    i_1< \cdots <i_a, \ j_1 < \cdots <j_b, \ m_1<\dots <m_d, \ l_1< l_2<\dots < l_c.
    $$

\end{enumerate}

Notice that in the last three types, the commutators appear and the first two variables are exactly of the required kind. 

We next show that these polynomials are linearly independent modulo $\I^*(A)$.
Consider the evaluation
$$
\mathcal{E}:
\begin{cases}
	y_{i_1}^{+}, \dots, y_{i_a}^{+} \leadsto e_{11} + e_{44} \\
	y_{j_1}^{-}, \dots, y_{j_b}^{-} \leadsto e_{11} - e_{44} \\
	z^{+} \leadsto e_{12} + e_{34} \\ 
	y_{r_1}^{+}, \dots, y_{r_e}^{+} \leadsto e_{22} + e_{33} \\
	y_{s_1}^{-}, \dots, y_{s_f}^{-} \leadsto e_{22} - e_{33}
\end{cases}.
$$
We get that  $y_{i_1}^{+}\dots y_{i_a}^{+} y_{j_1}^{-} \dots y_{j_b}^{-} z^{+} y_{r_1}^{+}\dots y_{r_e}^{+} y_{s_1}^{-} \dots y_{s_f}^{-}$ evaluates to $e_{12}$, while all the other polynomials evaluate to $0$.
By varying the indices we obtain that the coefficients of all the polynomials of type~\eqref{n1,n2,1,0:gen:D} are zero.

In order to show that all the other polynomials are linearly independent, it is sufficient to evaluate $z^+$ to $e_{12} + e_{34}$ and the $y$'s as in Case 1.

\medskip

Notice that the case $n_3 = 0$ and $n_4 = 1$ can be treated analogously, using identity $(ii)$ of Lemma \ref{identities hard of ut4}.

\medskip

\noindent \textbf{Case 3:} $n_3 + n_4 = 2$. 

\smallskip

Assume first $n_3 = 0 $ and $n_4 = 2$. In this case, we are dealing with polynomials of the kind
\begin{equation}  \label{polys n1 n2 0 2}
w_{\sigma(1)} \cdots w_{\sigma(i)} z_k^- w_{\sigma(i+1)} \cdots w_{\sigma(i+j)} z_h^- w_{\sigma(i+j+1)} \cdots w_{\sigma(n_1+n_2)}, \ \sigma \in S_{n_1+n_2}.
\end{equation}

According to the identities $(a)$, $(b)$ and $(c)$ of Lemma \ref{identities ut4 n3 + n4 = 2}, we can order the even variables in \eqref{polys n1 n2 0 2}, namely
\begin{equation*}
y_{i_1}^{+}\cdots y_{i_a}^{+}y_{i_{a+1}}^{-} \cdots y_{i_{a+a'}}^{-} z_k^{-} y_{j_1}^{+} \cdots y_{j_b}^{+} y_{j_{b+1}}^{-} \cdots y_{j_{b+b'}}^{-} z_h^{-} y_{l_1}^{+} \cdots y_{l_c}^{+}y_{l_{c+1}}^{-} \cdots y_{l_c}^{-},
\end{equation*}
where $i_1<\cdots <i_a$, $i_{a+1}<\cdots <i_{a+a'}$, $j_1<\cdots <j_b$, $j_{b+1}<\cdots <j_{b+b'}$, $l_1<\cdots <l_c$ and $l_{c+1}<\cdots <l_{c+c'}$. 

Finally, by identity $(f)$ of Lemma \ref{identities ut4 n3 + n4 = 2}, we can reduce a polynomial of type $P_1$ (see the polynomials presented before Lemma \ref{identities ut4 n3 + n4 = 2}), in which there are three groups of $y$'s, to a polynomial in which at most two groups of $y$'s appear (types $P_2$, $P_3$ and $P_4$).

In conclusion, we obtain the following types of generators:
\begin{enumerate}[(A)]
	\item\label{n1,n2,0,2:gen:A} $y_{i_1}^{+}\cdots y_{i_a}^{+}y_{i_{a+1}}^{-}\cdots y_{i_{a+a'}}^{-} z_k^{-} y_{j_1}^{+} \cdots y_{j_b}^{+} y_{j_{b+1}}^{-} \cdots y_{j_{b+b'}}^{-} z_h^{-}$,
    \vspace{0.1 cm}
	\item\label{n1,n2,0,2:gen:B} $z_k^{-} y_{l_1}^{+} \cdots y_{l_c}^{+}y_{l_{c+1}}^{-} \cdots y_{l_{c+c'}}^{-} z_h^{-} y_{m_1}^{+} \cdots y_{m_{d}}^{+}y_{m_{d+1}}^{-} \cdots y_{m_{d+d'}}^{-}$,
\end{enumerate}
where 
$i_1<\cdots <i_a$, $i_{a+1}<\cdots <i_{a+a'}$, $j_1<\cdots <j_b$, $j_{b+1}<\cdots <j_{b+b'}$, $l_1<\cdots <l_c$,  $l_{c+1}<\cdots <l_{c+c'}$, $m_1< \dots <m_d$, $m_{d+1}< \dots <m_{d+d'}$ and $d+d'\neq 0$.

We next show that these polynomials are linearly independent modulo $\I^*(A)$.
Consider the evaluation
$$
\mathcal{E}:
\begin{cases}
	y_{i_1}^{+}, \dots, y_{i_a}^{+} \leadsto e_{11} + e_{44} \\
	y_{i_{a+1}}^{-}, \dots, y_{i_{a+a'}}^{-} \leadsto e_{11} - e_{44} \\
    y_{j_1}^{+}, \dots, y_{j_b}^{+} \leadsto e_{22} + e_{33} \\
    y_{j_{b+1}}^{-}, \dots, y_{j_{b+b'}}^{-} \leadsto e_{22} - e_{33} \\
	z_k^{-} \leadsto e_{12} - e_{34} \\ 
	z_h^{-} \leadsto e_{23} \\ 
\end{cases}.
$$
We get that $y_{i_1}^{+}\dots y_{i_a}^{+}y_{i_{a+1}}^{-}\dots y_{i_{a+a'}}^{-} z_k^{-} y_{j_1}^{+} \dots y_{j_b}^{+} y_{j_{b+1}}^{-} \dots y_{j_{b+b'}}^{-} z_h^{-} $ evaluates to $ e_{13}$,
$z_h^- y_{j_1}^{+}\dots y_{j_b}^{+} y_{j_{b+1}}^{-} \cdots y_{j_{b+b'}}^{-} z_k^{-} y_{i_1}^{+} \dots y_{i_a}^{+} y_{i_{a+1}}^{-} \dots y_{i_{a+a'}}^{-} $ evaluates to $(-1)^{a^\prime+b^\prime+1}e_{24}$,
while all the other polynomials evaluate to $0$.
By varying the indices we obtain that the coefficients of all the above polynomials are zero.

\smallskip

Now consider $n_3 = n_4 = 1$. Here we proceed exactly as before, taking into account also the identities $(d)$ and $(e)$ of Lemma \ref{identities ut4 n3 + n4 = 2}. We obtain the following types of generators:

\begin{enumerate}[(A)]
	\item\label{n1,n2,1,1:gen:A} $y_{i_1}^{+} \cdots y_{i_a}^{+}y_{i_{a+1}}^{-} \cdots y_{i_{a+a'}}^{-} z_{m_1}^{+} z_{m_2}^{-} y_{j_1}^{+} \cdots y_{j_b}^{+} y_{j_{b+1}}^{-} \cdots y_{j_{b+b'}}^{-} $, 
    \vspace{0.1 cm}
    
	\item\label{n1,n2,1,1:gen:B} $y_{h_1}^{+}\cdots y_{h_c}^{+} y_{h_{c+1}}^{-}\cdots y_{h_{c+c'}}^{-} z_{s_1}^{-} z_{s_2}^{+} y_{l_1}^{+} \cdots y_{l_d}^{+}y_{l_{d+1}}^{-} \cdots y_{l_{d+d'}}^{-} $,
\end{enumerate}
where $i_1< \cdots <i_a$, $i_{a+1}< \cdots <i_{a+a'}$, $j_1< \cdots <j_b$, $j_{b+1}< \dots <j_{b+b'}$, $h_1< \cdots <h_c$, $h_{c+1}< \cdots <h_{c+c'}$, $l_1< \cdots <l_d$ and $l_{d+1}< \cdots <l_{d+d'}$.

We next show that these polynomials are linearly independent modulo $\I^*(A)$.
Consider the evaluation
$$
\mathcal{E}:
\begin{cases}
	y_{i_1}^{+}, \dots, y_{i_a}^{+} \leadsto e_{11} + e_{44} \\
	y_{i_{a+1}}^{-}, \dots, y_{i_{a+a'}}^{-} \leadsto e_{11} - e_{44} \\
        y_{j_1}^{+}, \dots, y_{j_b}^{+} \leadsto e_{22} + e_{33} \\
        y_{j_{b+1}}^{-}, \dots, y_{j_{b+b'}}^{-} \leadsto e_{22} - e_{33} \\
        z_{m_1}^{+} \leadsto e_{12} + e_{34} \\ 
	z_{m_2}^{-} \leadsto e_{23} \\ 
\end{cases}.
$$
It follows that $y_{i_1}^{+} \dots y_{i_a}^{+}y_{i_{a+1}}^{-} \dots y_{i_{a+a'}}^{-} z_{m_1}^{+} z_{m_2}^{-} y_{j_1}^{+} \dots y_{j_b}^{+} y_{j_{b+1}}^{-} \dots y_{j_{b+b'}}^{-} $ evaluates to $(-1)^{b^\prime}e_{13}$, whereas
$y_{j_1}^{+} \dots y_{j_b}^{+} y_{j_{b+1}}^{-} \dots y_{j_{b+b'}}^{-} z_{m_2}^{-} z_{m_1}^{+} y_{i_1}^{+} \dots y_{i_a}^{+}y_{i_{a+1}}^{-} \dots y_{i_{a+a'}}^{-} $ evaluates to $(-1)^{a^\prime}e_{24}$,
and all the other polynomials evaluate to $0$. By varying the indices we obtain that the coefficients of the polynomials of both type are zero.

Finally, according to identity~\eqref{id:24}, it is not necessary to consider the case $n_3 = 2$ and $n_4 = 0$.

\medskip

\noindent \textbf{Case 4:} $n_3 + n_4 = 3$. 

\smallskip

Assume first $n_3 = 0$ and $n_4 = 3$. By Lemma \ref{identities ut4 n3 + n4 = 3}, we have to consider just polynomials of the form
\begin{itemize}
    \item $z_1^{-}z_2^{-}z_3^{-} w_1 \cdots w_k$;
    \vspace{0.1 cm}
    \item $z_1^{-} z_2^{-} w_1 \cdots w_k z_3^{-} v_1 \cdots v_h$.
\end{itemize}

By identities $(a)$, $(b)$ and $(c)$ of Lemma \ref{identities ut4 n3 + n4 = 2}, we can order the $w$'s and $v$'s. Moreover, we use identity~\eqref{id:26} to order the $z^-$'s.
In conclusion, we obtain the following types of generators:
\begin{enumerate}[(A)]
	\item\label{n1,n2,0,3:gen:A} $z_{l_1}^{-}z_{m}^{-}z_{l_2}^{-} y_{i_1}^{+} \cdots y_{i_k}^{+}y_{i_{k+1}}^{-} \cdots y_{i_{k+k'}}^{-}$,
    \vspace{0.1 cm}
	\item\label{n1,n2,0,3:gen:B} $z_{j_1}^{-}z_{t}^{-} y_{h_1}^{+}\cdots y_{h_a}^{+} y_{h_{a+1}}^{-}\cdots y_{h_{a+a'}}^{-} z_{j_2}^{-} y_{r_1}^{+}\cdots y_{r_b}^{+}y_{r_{b+1}}^{-}\cdots y_{r_{b+b'}}^{-}$,
\end{enumerate}
where $l_1 < l_2$, $j_1 < j_2$, $i_1< \cdots <i_k$, $i_{k+1}< \cdots <i_{k+k'}$, $h_1< \cdots <h_a$, $h_{a+1}< \cdots <h_{a+a'}$, $r_1< \cdots <r_b$, $r_{b+1}< \cdots <r_{b+b'}$ and $a+a' \geq 1$.

We next show that these polynomials are linearly independent modulo $\I^*(A)$.
Consider the evaluation
$$
\mathcal{E}_1:
\begin{cases}
   	z_{l_1}^{-}, z_{l_2}^{-} \leadsto e_{12} - e_{34} \\ 
	z_m^{-} \leadsto e_{23} \\ 
	y_{i_1}^{+}, \dots, y_{i_k}^{+} \leadsto e_{11} + e_{44} \\
	y_{i_{k+1}}^{-}, \dots, y_{i_{k+k'}}^{-} \leadsto e_{11} - e_{44} \\
\end{cases}.
$$
We get that $z_{l_1}^{-}z_{m}^{-}z_{l_2}^{-} y_{i_1}^{+}\cdots y_{i_k}^{+} y_{i_{k+1}}^{-}\cdots y_{i_{k+k'}}^{-}$ evaluates to $(-1)^{k^\prime+1}e_{14}$, and all the other polynomials evaluate to $0$. By varying the indices we obtain that the coefficients of the polynomials of type~\eqref{n1,n2,0,3:gen:A} are zero.

By making the evaluation
$$
\mathcal{E}_2:
\begin{cases}
	z_{j_1}^{-}, z_{j_2}^{-} \leadsto e_{12} - e_{34} \\ 
	z_t^{-} \leadsto e_{23} \\ 
	y_{h_1}^{+}, \dots, y_{h_a}^{+} \leadsto e_{22} + e_{33} \\
        y_{h_{a+1}}^{-}, \dots, y_{h_{a+a'}}^{-} \leadsto e_{22} - e_{33} \\
        y_{r_1}^{+}, \dots, y_{r_d}^{+} \leadsto e_{11} + e_{44} \\
        y_{r_{d+1}}^{-}, \dots, y_{r_{d+d'}}^{-} \leadsto e_{11} - e_{44} \\
\end{cases},
$$
it follows that $z_{j_1}^{-}z_{t}^{-} y_{h_1}^{+}\cdots y_{h_a}^{+} y_{h_{a+1}}^{-}\cdots y_{h_{a+a'}}^{-} z_{j_2}^{-} y_{r_1}^{+}\cdots y_{r_b}^{+}y_{r_{b+1}}^{-}\cdots y_{r_{b+b'}}^{-}$ evaluates to $(-1)^{a'+b'+1}e_{14}$, and all the other polynomials evaluate to $0$. By varying the indices we obtain that the coefficients of the polynomials of type~\eqref{n1,n2,0,3:gen:B} are zero.

\medskip

Assume now $n_3 = 1$ and $n_4 = 2$. Since $z_1^- y_4 z_2^+ y_5 z_3^- \equiv 0$ and $z_1^- y_2 z_3^- y_4 z_5^+ - z_5^+ y_2 z_3^- y_4 z_1^- \equiv 0$, we just need to consider polynomials in which the $z^-$'s are ordered and the $z^+$ is never in the middle between them. Moreover, the $y$'s can be ordered according to Lemma \ref{identities ut4 n3 + n4 = 2}. Finally, taking into account also the identities of Lemma \ref{identities ut4 n3 + n4 = 3}, we just need to consider the following kind of generators: 

\begin{enumerate}[(A)]
	\item\label{n1,n2,1,2:gen:A} $z_{s_1}^{-}z_{s_2}^{-}z^{+} y_{i_1}^{+}\cdots y_{i_a}^{+}y_{i_{a+1}}^{-}\cdots y_{i_{a+a'}}^{-}$,
    \vspace{0.1 cm}
    
	\item\label{n1,n2,1,2:gen:B} $z^{+}z_{m_1}^{-}z_{m_2}^{-} y_{j_1}^{+}\cdots y_{j_b}^{+}y_{j_{b+1}}^{-}\cdots y_{j_{b+b'}}^{-}$,    
\vspace{0.1 cm}

    \item\label{n1,n2,1,2:gen:C} $z_{p_1}^{-}z_{p_2}^{-} y_{h_1}^{+}\cdots y_{h_c}^{+}y_{h_{c+1}}^{-}\cdots y_{h_{c+c'}}^{-} z^{+} y_{r_1}^{+}\cdots y_{r_d}^{+}y_{r_{d+1}}^{-}\cdots y_{r_{d+d'}}^{-}$,
\vspace{0.1 cm}

\item\label{n1,n2,1,2:gen:D} $z^{+} z_{q_1}^{-} y_{l_1}^{+}\cdots y_{l_e}^{+}y_{l_{e+1}}^{-}\cdots y_{l_{e+e'}}^{-} z_{q_2}^{-} y_{t_1}^{+}\cdots y_{t_f}^{+}y_{t_{f+1}}^{-}\cdots y_{t_{f+f'}}^{-}$,
\end{enumerate}
where $s_1 < s_2$, $m_1 < m_2$, $p_1 < p_2$, $q_1 < q_2$, $i_1< \cdots <i_a$, $i_{a+1}< \cdots <i_{a+a'}$, $j_1 < \cdots <j_b$, $j_{b+1} < \cdots <j_{b+b'}$, $h_1< \cdots <h_c$, $h_{c+1}< \cdots <h_{c+c'}$, $r_1< \cdots <r_d$, $r_{d+1}< \cdots <r_{d+d'}$, $l_1< \cdots <l_e$, $l_{e+1}< \cdots <l_{e+e'}$, $t_1< \cdots <t_f$, $t_{f+1}< \cdots <t_{f+f'}$, $c+c' \geq 1$ and $e+e' \geq 1$.

We next show that these polynomials are linearly independent modulo $\I^*(A)$.
Consider the evaluation
$$
\mathcal{E}_1:
\begin{cases}
   	z_{s_1}^{-} \leadsto e_{12} - e_{34} \\ 
        z_{s_2}^{-} \leadsto e_{23} \\
	z^{+} \leadsto e_{12} + e_{34} \\ 
	y_{i_1}^{+}, \dots, y_{i_a}^{+} \leadsto e_{11} + e_{44} \\
	y_{i_{a+1}}^{-}, \dots, y_{i_{a+a'}}^{-} \leadsto e_{11} - e_{44} \\ 
\end{cases}.
$$
It follows that $z_{s_1}^{-}z_{s_2}^{-}z^{+} y_{i_1}^{+}\cdots y_{i_a}^{+}y_{i_{a+1}}^{-}\cdots y_{i_{a+a'}}^{-}$ evaluates to $(-1)^{a'} e_{14}$, and all the other polynomials evaluate to $0$. By varying the indices we obtain that the coefficients of the polynomials of type~\eqref{n1,n2,1,2:gen:A} are zero.

In order to show that the polynomials of type~\eqref{n1,n2,1,2:gen:B} are zero it is sufficient to consider the evaluation
$$
\mathcal{E}_2:
\begin{cases}
   	z_{m_1}^{-} \leadsto e_{23} \\ 
        z_{m_2}^{-} \leadsto e_{12} - e_{34} \\
	z^{+} \leadsto e_{12} + e_{34} \\ 
	y_{j_1}^{+}, \dots, y_{j_b}^{+} \leadsto e_{11} + e_{44} \\
	y_{j_{b+1}}^{-}, \dots, y_{j_{b+b'}}^{-} \leadsto e_{11} - e_{44} \\
\end{cases}.
$$

Now make the evaluation
$$
\mathcal{E}_3:
\begin{cases}
   	z_{p_1}^{-} \leadsto e_{12} - e_{34} \\ 
        z_{p_2}^{-} \leadsto e_{23} \\
	z^{+} \leadsto e_{12} + e_{34} \\ 
	y_{h_1}^{+}, \dots, y_{h_c}^{+} \leadsto e_{22} + e_{33} \\
	y_{h_{c+1}}^{-}, \dots, y_{h_{c+c'}}^{-} \leadsto e_{22} - e_{33} \\   	
        y_{r_1}^{+}, \cdots, y_{r_d}^{+} \leadsto e_{11} + e_{44} \\ 
  	y_{r_{d+1}}^{-}, \dots, y_{r_{d+d'}}^{-} \leadsto e_{11} - e_{44} \\
\end{cases}.
$$
We get that $z_{p_1}^{-}z_{p_2}^{-} y_{h_1}^{+}\cdots y_{h_c}^{+}y_{h_{c+1}}^{-}\cdots y_{h_{c+c'}}^{-} z^{+} y_{r_1}^{+}\cdots y_{r_d}^{+}y_{r_{d+1}}^{-}\cdots y_{r_{d+d'}}^{-}$ evaluates to $(-1)^{c'+d'} e_{14}$, and all the other polynomials evaluate to $0$. By varying the indices we obtain that the coefficients of the polynomials of type~\eqref{n1,n2,1,2:gen:C} are zero.

As before, one can prove that the coefficients of the polynomials of type~\eqref{n1,n2,1,2:gen:D} are zero by using the evaluation
$$
\mathcal{E}_4:
\begin{cases}
   	z_{q_1}^{-} \leadsto e_{23} \\ 
        z_{q_2}^{-} \leadsto e_{12} - e_{34} \\
	z^{+} \leadsto e_{12} + e_{34} \\ 
	y_{l_1}^{+}, \dots, y_{l_e}^{+} \leadsto e_{22} + e_{33} \\
	y_{l_{e+1}}^{-}, \dots, y_{l_{e+e'}}^{-} \leadsto e_{22} - e_{33} \\
        y_{t_1}^{+}, \dots, y_{t_f}^{+} \leadsto e_{11} + e_{44} \\
        y_{t_{f+1}}^{-}, \dots, y_{t_{f+f'}}^{-} \leadsto e_{11} - e_{44} \\
\end{cases}.
$$

\medskip

Finally, let us consider the case $n_3 = 2$ and $n_4 = 1$. We just need to take into account identities~\eqref{id:24} and~\eqref{id:25} and proceed as in the previous case. We get the following kind of generators: 

\begin{enumerate}[(A)]
	\item\label{n1,n2,2,1:gen:A} $z_{s_1}^{+} z^{-} z_{s_2}^{+} y_{i_1}^{+}\cdots y_{i_a}^{+}y_{i_{a+1}}^{-}\cdots y_{i_{a+a'}}^{-}$, 
    \vspace{0.1 cm}
    
	\item\label{n1,n2,2,1:gen:B} $z_{m_1}^{+} z^{-}  y_{j_1}^{+}\cdots y_{j_b}^{+} y_{j_{b+1}}^{-}\cdots y_{j_{b+b'}}^{-} z_{m_2}^{+} y_{l_1}^{+}\cdots y_{l_c}^{+}y_{l_{c+1}}^{-}\cdots y_{l_{c+c'}}^{-}$,
\end{enumerate}
where $s_1 < s_2$, $m_1 < m_2$, $i_1< \cdots <i_a$, $i_{a+1}< \cdots <i_{a+a'}$, $j_1 < \cdots <j_b$, $j_{b+1} < \cdots <j_{b+b'}$, $l_1< \cdots <l_c$, $l_{c+1}< \cdots <l_{c+c'}$ and $b+b' \geq 1$.

We next show that these polynomials are linearly independent modulo $\I^*(A)$.
By making the evaluation
$$
\mathcal{E}_1:
\begin{cases}
   	z_{s_1}^{+}, z_{s_2}^{+} \leadsto e_{12} + e_{34} \\ 
        z^{-} \leadsto e_{23} \\ 
	y_{i_1}^{+}, \dots, y_{i_a}^{+} \leadsto e_{11} + e_{44} \\
	y_{i_{a+1}}^{-}, \dots, y_{i_{a+a'}}^{-} \leadsto e_{11} - e_{44} \\
\end{cases},
$$
we get that $z_{s_1}^{+} z^{-} z_{s_2}^{+} y_{i_1}^{+}\cdots y_{i_a}^{+} y_{i_{a+1}}^{-}\cdots y_{i_{a+a'}}^{-}$ evaluates to $(-1)^{a'} e_{14}$, and all the other polynomials evaluate to $0$. By varying the indices we obtain that the coefficients of the polynomials of type~\eqref{n1,n2,2,1:gen:A} are zero.

By making the evaluation
$$
\mathcal{E}_2:
\begin{cases}
   	z_{m_1}^{+}, z_{m_2}^{+} \leadsto e_{12} + e_{34} \\ 
        z^{-} \leadsto e_{23} \\ 
	y_{j_1}^{+}, \dots, y_{j_b}^{+} \leadsto e_{22} + e_{33} \\
	y_{j_{b+1}}^{-}, \dots, y_{j_{b+b'}}^{-} \leadsto e_{22} - e_{33} \\
        y_{l_1}^{+}, \dots, y_{l_c}^{+} \leadsto e_{11} + e_{44} \\
        y_{l_{c+1}}^{-}, \dots, y_{l_{c+c'}}^{-} \leadsto e_{11} - e_{44} \\
\end{cases},
$$
it follows that $z_{m_1}^{+} z^{-}  y_{j_1}^{+}\cdots y_{j_b}^{+} y_{j_{b+1}}^{-}\cdots y_{j_{b+b'}}^{-} z_{m_2}^{+} y_{l_1}^{+}\cdots y_{l_c}^{+}y_{l_{c+1}}^{-}\cdots y_{l_{c+c'}}^{-}  $ evaluates to $ (-1)^{b'+c'} e_{14}$, and all the other polynomials evaluate to $0$. By varying the indices we obtain that the coefficients of the polynomials of type~\eqref{n1,n2,2,1:gen:B} are zero.
\end{proof}

As a consequence of the explicit basis for the spaces $P_{n_1,\dots,n_4}$ modulo $\I^*(A)$ in the previous theorem, we have the next result.

\begin{Theorem}\label{codimensions of ut4}
Let $A$ be the $*$-algebra $UT_4(F)$ with elementary $\mathbb{Z}_2$-grading induced by $(0,1,0,1)$ and superinvolution $\bar{s}$. The $n$-th $*$-codimension of $A$ is given by
$$ 
c_n^*(A) = 2^{n}(n+2) - 2^{n-2}n(n-1) + 2^{2n-1} (n^2-2) + 2^{2n-5} n(n-1)(n-2).
$$
\end{Theorem}

\begin{proof}

Let us consider the four cases appearing in the proof of Theorem \ref{ideal of ut4}. We will thus take into account the subspaces $P_{n_1,\dots,n_4}$. For each one of them, we count the total number of generators for all $n_1,n_2,n_3,n_4$, where $n_3$ and $n_4$ are subject to the conditions defining each case.

    \medskip

    \noindent \textbf{Case 1:} $n_3 = n_4 = 0$. 

    \smallskip
    

It is immediate to see that the generators of type \eqref{n1,n2,0,0:gen:A} are exactly
$$ 
\sum_{a=0}^{n} \binom{n}{a} = 2^n.
$$ 

Now let us focus our attention on the generators of type \eqref{n1,n2,0,0:gen:B}: 
$$
y_{i_1}^{+} \dots y_{i_a}^{+}y_{j_1}^{-}\dots y_{j_b}^{-}[y_{\kappa}^{+}, y_{l_1}^{+}, \dots, y_{l_c}^{+}, y_{m_1}^{-}, \dots, y_{m_d}^{-}],
$$ 
where $ \kappa >l_1<\dots < l_c, \ m_1<\dots <m_d$. 

Fix $l$ to be the number of variables inside the commutator and $k := 1 + c$ the number of symmetric even variables inside the commutator. 

We have $\binom{n}{a}$ choices for the indices of the symmetric even variables outside the commutator, and then $\binom{n-a}{n-a-l}$ choices for the indices of the skew even variables outside the commutator. We multiply these by the number of choices inside the commutator, that is: for each possible choice for the number $k$ of symmetric even variables we can choose the indices in $\binom{l}{k}$ ways, and we can order the variables inside in $k-1$ ways.

In conclusion the generators of type \eqref{n1,n2,0,0:gen:B} are
$$ 
\mathcal{S}_1(n) := \sum_{l=2}^{n} \left( \sum_{a=0}^{n-l} \binom{n}{a} \binom{n-a}{n-a-l} \cdot \sum_{k=2}^{l} \binom{l}{k}(k-1) \right). 
$$
    
The number of generators of type \eqref{n1,n2,0,0:gen:C} is exactly the same. With a similar approach one easily get that there are
$$ 
\mathcal{S}_2(n) := \sum_{l=2}^{n} \left( \sum_{a=0}^{n-l} \binom{n}{a} \binom{n-a}{n-a-l} \cdot \sum_{k=1}^{l-1} \binom{l}{k} \right) 
$$
generators of type \eqref{n1,n2,0,0:gen:D}.

Using standard combinatorial arguments, it is easily seen that
    \begin{align*}
        \mathcal{S}_1(n) & = (n-4)2^{2n-2} + 3^{n}, \\
        \mathcal{S}_2(n) & = 2^{2n} + 2^{n} - 2 \cdot 3^{n}.
    \end{align*}
    
Therefore, the number $\mathcal{S}(n)$ of generators in Case 1 is 
\begin{equation*}
\begin{split}
\mathcal{S}(n) &= 2^{n} + 2 \mathcal{S}_1(n) + \mathcal{S}_2(n) \\
&= 2^{n+1}\left(1+2^{n-2}(n-2)\right).
\end{split}
\end{equation*}

    \noindent \textbf{Case 2:} $n_3 + n_4 = 1$.  

    \smallskip

Let us deal first with the generators of type \eqref{n1,n2,1,0:gen:D}, ~\eqref{n1,n2,1,0:gen:A}, \eqref{n1,n2,1,0:gen:B} and \eqref{n1,n2,1,0:gen:C} of the sub-case $(n_1, n_2, 1, 0)$. The other sub-case, namely $(n_1, n_2, 0, 1)$, is completely analogous.

There are 
    $$
    \mathcal{U}_1(n) := n \sum_{a=0}^{n-1} \binom{n-1}{a} \sum_{b=0}^{n-1-a} \binom{n-1-a}{b} \sum_{c=0}^{n-1-a-b} \binom{n-1-a-b}{c} 
    $$
    generators of type~\eqref{n1,n2,1,0:gen:D}. Moreover, the number of generators of types~\eqref{n1,n2,1,0:gen:A}, \eqref{n1,n2,1,0:gen:B} and~\eqref{n1,n2,1,0:gen:C} can be computed similarly to Case 1, more precisely they are
    $$ \mathcal{U}_2(n) := n(2 \cdot \mathcal{S}_1(n-1) + \mathcal{S}_2(n-1)).$$
    
    Again by standard combinatorial arguments, it follows that
    \begin{align*}
        \mathcal{U}_1(n) & = n 4^{n-1}, \\
        \mathcal{U}_2(n) & = n(2^{2n-3}(n-3) + 2^{n-1}).
    \end{align*}
    
    In conclusion, the number $\mathcal{U}(n)$ of generators in Case 2 is
    $$\mathcal{U}(n) = n2^{n} \left(1+(n-1)2^{n-2}\right).$$

    \noindent \textbf{Case 3.} $n_3 + n_4 = 2$.   

    \smallskip

    Consider the generators of type~\eqref{n1,n2,1,1:gen:A} and~\eqref{n1,n2,1,1:gen:B} of the sub-case $(n_1, n_2, 1, 1)$.

    The number of generators of type~\eqref{n1,n2,1,1:gen:A} is 
    $$ \mathcal{V}_1(n) := n (n-1) \sum_{a=0}^{n-2} \binom{n-2}{a} \sum_{b=0}^{n-2-a} \binom{n-2-a}{b} \sum_{c=0}^{n-2-a-b} \binom{n-2-a-b}{c}.$$
    
    Observe that $\mathcal{V}_1(n) = n \cdot \mathcal{U}_1(n-1)$. Hence the generators of type~\eqref{n1,n2,0,2:gen:A} are
    $$ \mathcal{V}_1(n) = n(n-1)2^{2n-4}. $$
    
    We double this number for the generators of type~\eqref{n1,n2,1,1:gen:B}, which can be obtained with the same computation. 
    
    Now, consider the sub-case $(n_1, n_2, 0, 2)$. The number of generators of type~\eqref{n1,n2,0,2:gen:A} is exactly as above. The generators of type~\eqref{n1,n2,0,2:gen:B} are subject to the condition $d+d'\neq 0$, therefore they are 
    $$ \mathcal{V}_1(n) - n(n-1)\sum_{c=0}^{n-2}\binom{n-2}{c} =\mathcal{V}_1(n) - n(n-1)2^{n-2}.$$
    In total, the number $\mathcal{V}(n)$ of generators in Case 3 is
    $$ \mathcal{V}(n) = 4\mathcal{V}_1(n) - n(n-1)2^{n-2} = n(n-1)2^{n-2}(2^{n}-1).$$

    \noindent \textbf{Case 4:} $n_3 + n_4 = 3$.   

    \smallskip
    
    Consider the generators of type~\eqref{n1,n2,0,3:gen:A} and~\eqref{n1,n2,0,3:gen:B} in the sub-case $(n_1, n_2, 0, 3)$.

    Note that, when $a+a'=0$, the generators of type~\eqref{n1,n2,0,3:gen:B} fall into type~\eqref{n1,n2,0,3:gen:A}, so we can treat them all together. They are
    $$ 
    \mathcal{Z}_1(n) := n \binom{n-1}{2} \sum_{a=0}^{n-3} \binom{n-3}{a} \sum_{b=0}^{n-3-a} \binom{n-3-a}{b} \sum_{c=0}^{n-3-a-b} \binom{n-3-a-b}{c}.
    $$
    
    Notice that
    $$ \mathcal{Z}_1(n) = \frac12 n \mathcal{V}_1(n-1) = n (n-1)(n-2)2^{2n-7}.$$

    The sub-case $(n_1, n_2, 2, 1)$ is completely analogue. 
    
    Finally, let us deal with the sub-case $(n_1, n_2, 1, 2)$, where we have to consider the generators of type~\eqref{n1,n2,1,2:gen:A}, \eqref{n1,n2,1,2:gen:B}, \eqref{n1,n2,1,2:gen:C} and~\eqref{n1,n2,1,2:gen:D}. We pair the generators of type~\eqref{n1,n2,1,2:gen:A} with those of type~\eqref{n1,n2,1,2:gen:C}, and the generators of type~\eqref{n1,n2,1,2:gen:B} with those of type~\eqref{n1,n2,1,2:gen:D}. We note that each pair can be counted as above. 

    Therefore, in total, the number $\mathcal{Z}(n)$ of generators in Case 4 is
    $$ \mathcal{Z}(n) = 4\mathcal{Z}_1(n) = n (n-1)(n-2)2^{2n-5}.$$

    Finally, summing up all the numbers, we obtain that
    \begin{align*}
    c_n^*(A) &= \mathcal{S}(n) + \mathcal{U}(n) + \mathcal{V}(n) + \mathcal{Z}(n) \\
    &= 2^{n}(n+2) - 2^{n-2}n(n-1) + 2^{2n-1} (n^2-2) + 2^{2n-5} n(n-1)(n-2).
    \end{align*}
\end{proof}

By the above theorem, and also according to \cite{Ioppolo2018}, it immediately follows that
\[
\exp^\ast(A) = \lim_{n \to +\infty} \sqrt[n]{ c_n^{\ast}(A)} = 4.
\]

We would like to point out that the remaining cases, namely, \( UT_4(F) \) with the elementary grading \( (0,1,0,1) \) and the super-reflection superinvolution, as well as \( UT_4(F) \) with the other elementary gradings of $*$-type and the corresponding superinvolutions, can be treated in a similar fashion. Consequently, it is certainly possible to determine the generators of the identities and the codimensions in these cases. However, due to the considerable computational complexity involved, already evident in the case treated here, we have chosen not to pursue these computations.

\section{\bf Images of polynomials} 

Let $A$ be a $*$-algebra. Every polynomial in $F\langle Y\cup Z,* \rangle$ induces a function on $A$ as follows: if $f=f(y_1^+,\dots,y_n^+,y_1^-\dots,y_m^-,z_1^+, \ldots,z_t^+,z_1^-,\dots,z_s^-)$, then $\varphi_f$ is the map
\[
\varphi_f \colon \left(A_0^+ \right)^n\times \left(A_0^- \right)^m\times \left(A_1^+ \right)^t \times \left(A_1^- \right)^s \longrightarrow  A
\]
which sends the element ${\bf a} = (a_1^+,\dots,a_n^+,a_1^-\dots,a_m^-,b_1^+,\dots,b_t^+,\dots,b_1^-,\dots,b_s^-)$ to the evaluation of $f$ on this tuple, i.e., 
\[
\varphi_f({\bf a}) =  f(a_1^+,\dots,a_n^+,a_1^-\dots,a_m^-,b_1^+,\dots,b_t^+,\dots,b_1^-,\dots,b_s^-).
\]

\begin{Definition}
The image of a polynomial $f\in F\langle Y\cup Z,* \rangle$ on some $*$-algebra $A$ is defined as the image of the function $\varphi_f$.
\end{Definition}

In light of the L'vov-Kaplansky conjecture we are interested in studying the images of polynomials which induces multilinear functions on $*$-algebras. Recall from the second section that the space of multilinear polynomials in $F\langle Y\cup Z,* \rangle$ was defined as
\[
P_n^*=\mbox{span}_F \{w_{\sigma(1)}\cdots w_{\sigma(n)} \ | \ \sigma \in S_n, \  w_i \in \{y_i^+,y_i^-,z_i^+,z_i^-\}, \ i=1,\dots,n\}.
\]

Notice that in $P_n^*$ there exist polynomials, such as $y_1^+y_2^++y_1^+y_2^-$, which, in general, do not induce multilinear functions on a $*$-algebra. For this reason, we will restrict ourselves to study images of multilinear polynomials from the following subspaces of $P_n^*$. Take $m\in\mathbb{N}$ and let $\Lambda_m=\{w_1,\dots,w_m\}$ be a subset of $Y^+ \cup Y^- \cup Z^+ \cup Z^-$. We set 
\[
P_{\Lambda_m}=\mbox{span}_F \left \{w_{\sigma(1)} \cdots w_{\sigma(m)} \ | \ \sigma\in S_m, \ w_i\in \Lambda_m, \ i=1,\dots,m \right \}.
\]
In others words, the elements from $P_{\Lambda_m}$ are linear combinations of multilinear words in $m$ fixed letters from the alphabet $Y\cup Z$.

\medskip
The main goal of this section is to give a complete classification of the images of the polynomials in $P_{\Lambda_m}$ on the $*$-algebras of upper triangular matrices $UT_2(F)$ and $UT_3(F)$, endowed with any possible superinvolution. According to Theorem \ref{superinvolutions on UTn}, we assume that our algebras are endowed with an elementary $\mathbb{Z}_2$-grading of $*$-type, i.e., induced by a sequence $(g_1,\dots,g_n)$ such that $g_1+g_n=g_2+g_{n-1}=\cdots=g_n+g_1$. 

\medskip

Let us start by considering the algebra $UT_2(F)$. Up to isomorphism, there are just two elementary gradings of $*$-type on such an algebra, namely $(0,0)$ and $(0,1)$. Moreover, since the odd part in both gradings is nilpotent of nilpotency index $2$, it immediately follows that the superinvolutions on $UT_2(F)$ are actually graded involutions. From \cite[Theorem 3.1]{fagundes}, we get the following result.

\begin{Theorem}
Let $A$ be the algebra $UT_2(F)$ with elementary grading $(0,0)$ or $(0,1)$ and endowed with any possible superinvolution. Then the images of multilinear $*$-polynomials on $A$ is always a homogeneous vector space.     
\end{Theorem}

Now let us focus our attention on $3\times 3$ upper triangular matrices. In  this case, there are again just $2$ elementary gradings of $*$-type, namely: 
\[
(0,0,0) \ \text{and} \ (0,1,0). 
\]

The trivial grading case also follows from \cite{fagundes}, since in this situation a superinvolution is just an involution. Here we actually have that the image of a multilinear polynomial is not always a vector space. Indeed, the statement of \cite[Proposition 4.1]{fagundes}, can be written, in our language, as follows.

\begin{Theorem}
Let $A$ be the algebra $UT_n(F)$, $n \ge 3$, with trivial grading $(0,\ldots,0)$ and endowed with the reflection (super)involution $\circ$. Then the image of the multilinear $*$-polynomial $y_1^- y_2^-$ on $A$ is not a vector space.   
\end{Theorem}

Let us finally investigate the images of multilinear polynomials on the algebra $UT_3(F)$ with elementary grading $(0,1,0)$ and endowed with the only superinvolution that can be defined in this case, namely the super-reflection $\bar{\circ}$ given by:
\[
\begin{pmatrix}
    a&b&c\\
    0&d&e\\
    0&0&f
\end{pmatrix}^{\bar{\circ}}=\begin{pmatrix}
    f&e&-c\\
    0&d&b\\
    0&0&a
\end{pmatrix}.
\]

It is easy to see that such a $*$-algebra, can be decomposed in the following way: $UT_3(F) = A_0^+\oplus A_0^-\oplus A_1^+\oplus A_1^-$, where
\begin{equation*}
\begin{split}
A_0^+ &= \mbox{span}_F \{e_{11}+e_{33},e_{22}\}, \\ 
A_0^- &= \mbox{span}_F \{e_{11}-e_{33},e_{13}\}, \\
A_1^+ &= \mbox{span}_F \{e_{12}+e_{23}\}, \\
A_1^- &= \mbox{span}_F \{e_{12}-e_{23}\}.    
\end{split}
\end{equation*}

Let us start by considering the case in which $f\in P_{\Lambda_m}$ is a multilinear polynomial on even variables only. 

Consider a set of commutative variables 
\[
\Xi= \left \{\xi_{ij}^{(k)} \ | \ i,j=1,2,3, \ k = 1,2,\dots \right \}
\]
and the free associative and commutative algebra $F[\Xi]$. Moreover we will also consider the following evaluations of the even variables
\[
\overline{y_k^+}=\begin{pmatrix}
    \xi_{11}^{(k)}&0&0\\
    0&\xi_{22}^{(k)}&0\\
    0&0&\xi_{11}^{(k)}
\end{pmatrix} 
\quad \text{and} \quad 
\overline{y_k^-}=\begin{pmatrix}
    \xi_{11}^{(k)}&0&\xi_{13}^{(k)}\\
    0&0&0\\
    0&0&-\xi_{11}^{(k)}
\end{pmatrix}.
\] 

We need a few technical lemmas. In what follows we use the hat \ $\widehat{ }$ \ over a variable to indicate that it is omitted from the expression.

\begin{Lemma}\label{productsym}
The product $y_1^+\cdots y_l^+$ is evaluated to
    \[ 
    \begin{pmatrix} \displaystyle
        \prod_{i=1}^l \xi_{11}^{(i)}&0&0\\
        0& \displaystyle \prod_{i=1}^l \xi_{22}^{(i)}&0\\
        0&0& \displaystyle \prod_{i=1}^l \xi_{11}^{(i)}
    \end{pmatrix}.
    \]
\end{Lemma}
\begin{proof}
    The proof is straightforward and hence omitted.
\end{proof}

\begin{Lemma}\label{commutatorsskew}
If $m-l\geq 2$, then $\left [y_i^-,y_{l+1}^-,\dots,\widehat{y_i^-},\dots,y_m^- \right]$ is evaluated to
    \[
        (-1)^{m-l}2^{m-l-1} \left( \xi_{11}^{(i)}\xi_{13}^{(l+1)}-\xi_{13}^{(i)}\xi_{11}^{(l+1)} \right) \xi_{11}^{(l+2)}\cdots \widehat{\xi_{11}^{(i)}}\cdots \xi_{11}^{(m)}e_{13}.
    \]
\end{Lemma}
\begin{proof}
The proof is by induction on $m-l$. In case $m-l=2$, it is clear that
    \[
    \left[y_{l+2}^-,y_{l+1}^- \right]=2 \left( \xi_{11}^{(l+2)}\xi_{13}^{(l+1)}-\xi_{13}^{(l+2)}\xi_{11}^{(l+1)} \right) e_{13}
    \]
    which is in the desired form.

Now assume that the lemma is true for $m-l=k$ and let us prove it for $m-l=k+1$. In this case,  $m=l+k+1$. In order to simplify the exposition, write
\[
{\bf b} := \left( \xi_{11}^{(i)}\xi_{13}^{(l+1)}-\xi_{13}^{(i)}\xi_{11}^{(l+1)} \right) \xi_{11}^{(l+2)}\cdots \widehat{\xi_{11}^{(i)}}\cdots \xi_{11}^{(l+k)}.
\]
It follows that
\begin{equation*}
    \begin{split}
\left[y_i^-,y_{l+1}^-,\dots,\widehat{y_i^-},\dots,y_{l+k+1}^- \right] &=\left[ \left[ y_i^-,y_{l+1}^-,\dots,\widehat{y_i^-},\dots,y_{l+k}^- \right],y_{l+k+1}^- \right]  \\
    &= \left[ (-1)^{m-l-1}2^{m-l-2} {\bf b} e_{13},y_{l+k+1}^- \right] \\
    &= - (-1)^{m-l-1}2^{m-l-2} {\bf b} \xi_{11}^{(l+k+1)}e_{13} - (-1)^{m-l-1}2^{m-l-2} {\bf b} \xi_{11}^{(l+k+1)}e_{13} \\
    &= (-1)^{m-l}2^{m-l-1} \textbf{b} \xi_{11}^{(l+k+1)}e_{13},
    \end{split}
\end{equation*}
and the proof is completed.
\end{proof}

\begin{Lemma}\label{productskew}
    The product $y_{l+1}^-\cdots y_{m}^-$ is evaluated to 
    \[
    \begin{pmatrix}
        \displaystyle\prod_{i=l+1}^m \xi_{11}^{(i)}&0&\displaystyle\sum_{i=1}^k (-1)^{k+i}\xi_{11}^{(l+1)}\cdots \xi_{13}^{(l+i)}\cdots \xi_{11}^{(l+k)}\\
        0&0&0\\
        0&0&(-1)^{k}\displaystyle\prod_{i=l+1}^m \xi_{11}^{(i)}
    \end{pmatrix},
    \]
    where $k\in\mathbb{N}$ is such that $m=l+k$.
\end{Lemma}
\begin{proof}
    Once again we proceed by induction on $k=m-l$. The base of the induction ($m=l+1$) is trivial. So let us assume the lemma is true for $m-l=k$ and prove it for $m-l=k+1$. In order to simplify the exposition, write
    \[
    {\bf S}:=  \displaystyle\sum_{i=1}^{k}(-1)^{k+i+1}\xi_{11}^{(l+1)}\cdots \xi_{13}^{(l+i)}\cdots \xi_{11}^{(l+k)}.
    \]
    
We obtain that
\begin{equation*}
    \begin{split}
     y_{l+1}^-\cdots y_m^- &= (y_{l+1}^-\cdots y_{l+k}^-)y_{l+k+1}^-\\
    &= \begin{pmatrix}
\displaystyle\prod_{i=l+1}^{l+k}\xi_{11}^{(i)}&0&\displaystyle\sum_{i=1}^k (-1)^{k+i} \xi_{11}^{(l+1)}\cdots \xi_{13}^{(l+i)}\cdots \xi_{11}^{(l+k)}\\
0&0&0\\
0&0&(-1)^k\displaystyle\prod_{i=l+1}^{l+k}\xi_{11}^{(i)}
    \end{pmatrix}y_{l+k+1}^-\\
    &= \begin{pmatrix}
\displaystyle\prod_{i=l+1}^{l+k}\xi_{11}^{(i)}\xi_{11}^{(l+k+1)}& 0 & \displaystyle \left( \prod_{i=l+1}^{l+k}\xi_{11}^{(i)} \right) \xi_{13}^{(l+k+1)} + {\bf S} \xi_{11}^{(l+k+1)}\\
0&0&0\\
0&0&(-1)^{k+1}\displaystyle\prod_{i=l+1}^{l+k}\xi_{11}^{(i)}\xi_{11}^{(l+k+1)}
    \end{pmatrix} \\
    &= \begin{pmatrix}
\displaystyle\prod_{i=l+1}^{l+k+1}\xi_{11}^{(i)}& 0 &\displaystyle \sum_{i=1}^{k+1}(-1)^{k+i+1}\xi_{11}^{(l+1)}\cdots \xi_{13}^{(l+i)}\cdots \xi_{11}^{(l+k+1)}\\
0&0&0\\
0&0&(-1)^{k+1}\displaystyle\prod_{i=l+1}^{l+k+1}\xi_{11}^{(i)}
\end{pmatrix}.
    \end{split}
\end{equation*}
\end{proof}

Finally, we recall the following result, which is a consequence of \cite[Lemma 3]{malev}.

\begin{Corollary}\label{dimupto2issubspace}
    Assume that the image of a multilinear polynomial $f$ on some algebra $A$ is contained in some $2$-dimensional vector space. Then, $f(A)$ is a vector subspace.
\end{Corollary}

Now we are in a position to prove the main result of this section.

\begin{Theorem}
Let $f\in P_{\Lambda_m}$. Then the image of $f$ on $UT_3(F)$ endowed with the elementary $\mathbb{Z}_2$-grading $(0,1,0)$ and the super-reflection superinvolution $\bar{\circ}$ is a vector space. Moreover, the image can be explicitly computed.
\end{Theorem}
\begin{proof}
Assume first that $f$ is a polynomial only in even variables, that is,
\[
f=f \left( y_1^+,\dots,y_l^+,y_{l+1}^-,\dots,y_m^- \right).
\]

By Theorem \ref{identities of ut3} and according to \cite{IoppoloMartino2018}, we can write $f$ as 
\[
f=\alpha y_1^+ \dots y_l^+y_{l+1}^- \dots y_m^-+\sum_{i=l+2}^m \alpha_i y_1^+\cdots y_l^+[y_i^-,y_{l+1}^-,\dots,\widehat{y_i^-},\dots,y_m^-].
\]

If $\alpha=0$, then by Lemma \ref{commutatorsskew} we see that $f(UT_3(F))\subset \mbox{span}_F \{e_{13}\}$. Since the image of a multilinear polynomial is closed under scalar product we must therefore have that $f(UT_3(F))=\{0\}$ or $f(UT_3(F))= \mbox{span}_F \{e_{13}\}$.

Now assume that $\alpha\neq0$ and let $k=m-l$. 

If $k=0$, then by Lemma \ref{productsym}, the image of $f$ on $UT_3(F)$ is exactly $A_0^+$. 

In case $k=1$, taking into account the $*$-identities from Theorem \ref{identities of ut3}, we must have $f=\alpha y_1^+\cdots y_l^+y_m^-$. By Lemma \ref{productsym} any matrix from $A_0^+$ can be realized as a product of even symmetric ones. Moreover, since 
\[
(\lambda_1(e_{11}+e_{33})+\lambda_2e_{22})(\lambda_3(e_{11}-e_{33})+\lambda_4 e_{13})=\lambda_1(\lambda_3(e_{11}-e_{33})+\lambda_4 e_{13}),
\]
we can see that the image of $f$ on $UT_3(F)$ must be exactly $A_0^-$ in this case.

From now on we may assume that $k\geq 2$. By Lemmas \ref{productsym} and \ref{productskew}, we can see that the main diagonal of $f$ is given by
\[
\alpha\prod_{i=1}^{m}\xi_{11}^{(i)}e_{11}+\alpha(-1)^{m-l}\prod_{i=1}^m\xi_{11}^{(i)}e_{33}.
\]
On the other hand the entry $(1,3)$ of $f$ can be calculated through Lemmas \ref{productsym}, \ref{commutatorsskew} and \ref{productskew}, and it is equal to
\[
\prod_{j=1}^l\xi_{11}^{(j)}\sum_{i=1}^{k}\beta_{l+i} \xi_{11}^{(l+1)}\cdots \xi_{13}^{(l+i)}\cdots \xi_{11}^{(l+k)}, 
\]
where $k=m-l$ and
\begin{itemize}
\item[$\bullet$] $\displaystyle \beta_{l+1}=(-1)^{k+1}\alpha +\sum_{i=2}^{k}\alpha_{l+i}2^{k-1}(-1)^k$,
\vspace{0.1 cm}
\item[$\bullet$] $\beta_{l+i}=(-1)^{k+i}\alpha -(-1)^{k}2^{k-1}\alpha_{l+i}$, if $i > 1$.
\end{itemize}

Suppose that $\beta_{l+i}=0$, for any $i=1,\dots,k$. Then $f(UT_3(F))$ is a subset of the subspace $\mathcal{D}_k$ of matrices of the form
\[
\begin{pmatrix}
    a&0&0\\
    0&0&0\\
    0&0&(-1)^ka
\end{pmatrix}.
\]
Since $\alpha\neq0$, then $f(UT_3(F))\neq\{0\}$ and therefore we conclude that
\[
f(UT_3(F))=\mathcal{D}_k.
\]

In case $\beta_{l+j}\neq0$ for some $j\in\{1,\dots,k\}$, then $f(UT_3(F))$ is contained in the subspace $J_k$ of matrices of the form
\[
\begin{pmatrix}
    a&0&b\\
    0&0&0\\
    0&0&(-1)^ka
\end{pmatrix}.
\]

Now let us set 
\begin{equation*}
    \begin{split}
\xi_{11}^{(t)}   &=1, \ t = 1, \ldots, l, \\ 
\xi_{13}^{(l+j)} &=\beta_{l+j}^{-1}b, \\ 
\xi_{11}^{(l+j)} &=\alpha^{-1}a, \\
\xi_{11}^{(l+i)} &=1, \mbox{ for }  i\neq j. \\
\xi_{13}^{(l+i)} &=0, \mbox{ for }  i\neq j.        
    \end{split}
\end{equation*}

We obtain an arbitrary matrix from $J_k$ in the image of $f$ on $UT_3(F)$. Hence 
\[
f(UT_3(F))=J_k.
\]

Finally, we have to consider the case in which $f$ may contain odd variables. 

Recall that $z_1z_2z_3=0$ is a $*$-identity of $UT_3(F)$, where the $z_i$'s denote any symmetric or skew-symmetric variable of homogeneous degree $1$. Hence we can assume that $f$ has at most two odd variables. 

Clearly, the image of $z_1z_2$ on $UT_3(F)$ is a subset of $ \mbox{span}_F \{e_{13}\}$ and so it is completely determined. 

Assume that $f$ has exactly one odd variable. In this case we must have that
\[
f(UT_3(F))\subset A_1.
\]
Since $\dim A_1=2$, by Corollary \ref{dimupto2issubspace}, it follows that $f(UT_3(F))$ is a vector space. 

If $f(UT_3)$ has dimension $0$ or $2$, then we know exactly what is the image. Moreover it is easy to see that $A_1$ can be realized as the image on $UT_3(F)$ of the polynomial
\[
f(y_1^+,z_1^+)=y_1^+z_1^+.
\]

On the other hand, every $1$-dimensional subspace of $A_1$ can be realized as the image on $UT_3(F)$ of some multilinear polynomial. Indeed, let 
\[
V=\mbox{span}_F \{\alpha e_{12}+\beta e_{23}\}
\]
be an $1$-dimensional subspace of $A_1$ ($\alpha$ and $\beta$ are fixed scalars). For 
\[
f(y_1^-,z_1^+)=\alpha y_1^-z_1^+-\beta z_1^+y_1^-
\]
we have that $f(UT_3(F))=V$.
\end{proof}

Motivated by the behavior observed in the case of graded algebras with graded involution (see \cite{fagundes}), in the final part of this paper we investigate the following question: does there exist a $\mathbb{Z}_2$-grading on $UT_n(F)$ (for $n \geq 4$) and a multilinear polynomial whose image is not a vector space?

\medskip

Assume first $n$ even. Consider the $\mathbb{Z}_2$-grading on $UT_n(F)$ given by the sequence 
\[
(0,1,0,1,\dots,0,1).
\]
Since the above sequence is of $*$-type, the super-reflection $\overline{\circ}$ is a superinvolution on this algebra.
Now let us consider matrix units $e_{ij}$'s, with $i\leq j$ and write $j=i+k$, for some $k \in \{0,\dots,n-i\}$. According to the action of the super-reflection on such matrices, we get that

\begin{equation*}
\begin{split}
A_0^+&= \mbox{span} \left \{ e_{i,i+k}+e_{n+1-i-k,n+1-i} \ | \  k \equiv 0 \ (\mbox{mod} \; 4) , i=1,\dots,(n-k)/2 \right \} \\ 
&  \cup \, \mbox{span} \left \{e_{i,i+k} -e_{n+1-i-k,n+1-i} \ | \ k \equiv 2 \ (\mbox{mod} \; 4), i=1,\dots,(n-k)/2 \right \}, \\
A_1^+&= \mbox{span} \left \{ e_{i,i+k}+e_{n+1-i-k,n+1-i} \ | \  k \equiv 1 \ (\mbox{mod} \; 4) , i=1,\dots,(n-k+1)/2 \right \} \\ 
&  \cup \, \mbox{span} \left \{ e_{i,i+k}-e_{n+1-i-k,n+1-i} \ | \ k \equiv 3 \ (\mbox{mod} \; 4), i=1,\dots,(n-k-1)/2 \right \}.
\end{split}    
\end{equation*}

Now let us consider the following polynomial
\[
f(y^+,z^+)=y^+z^+.
\] 
According to the description of the symmetric even and odd parts of $A$ given above, we notice that $e_{11}+e_{nn}\in A_0^+$ and $e_{12}+e_{n-1,n}\in A_{1}^+$. Hence 
$$
e_{12}=(e_{11}+e_{nn})(e_{12}+e_{n-1,n})\in f \left(UT_n(F) \right).
$$ 
Analogously, since $e_{22}+e_{n-1,n-1}\in A_{0}^+$ and $e_{23}+e_{n-2,n-1}\in A_1^+$, their product produces $e_{23}\in f(UT_n(F))$. 

We now proceed as in \cite[Proposition 4.8]{fagundes}. If $f(UT_n(F))$ is a vector subspace, then $e_{12}+e_{23}\in f(UT_n(F))$. 

Let us write $e_{12}+e_{23}=AB$, for some $A\in A_0^+$ and $B\in A_1^+$. Write 
$$
A=\sum_{i,j=1}^n a_{ij}e_{ij} \ \ \ \mbox{ and } \ \ \ B=\sum_{i,j=1}^n b_{ij}e_{ij}.
$$
Notice that the entries $(1,2)$ and $(2,3)$ of $AB$ are equal to 
$$
\sum_{l=1}^n a_{1l}b_{l2}=a_{11}b_{12} \ \ \ \mbox{and} \ \ \ \sum_{l=1}^n a_{2l}b_{l3}=a_{22}b_{23}.
$$ 
Finally, the entry $(n-1,n)$ of $AB=e_{12}+e_{23}$ is $0$ and on the other hand, equals
$$
\sum_{l=1}^n a_{n-1,l}b_{ln}=a_{n-1,n-1}b_{n-1,n}.
$$ 
In conclusion, we get $a_{n-1,n-1}=a_{22}=0$ or $b_{n-1,n}=b_{12}=0$, which contradicts the assumption that $e_{12}+e_{23}\in f(UT_n(F))$.

\medskip

Here we highlight that the case $n$ odd can be treated analogously, by considering the $\mathbb{Z}_2$-grading given by the sequence $(0,1,0,1,\dots,0,1,0)$. So we have proved the following general result.

\begin{Theorem}
Let $UT_n(F)$, $n \ge 4$, with elementary $\mathbb{Z}_2$-grading given by $(0,1,0,1, \ldots)$ and endowed with the super-reflection superinvolution $\bar{\circ}$. Then the image of the polynomial $f(y^+,z^+)=y^+z^+$ is not a vector space.
\end{Theorem}

\end{document}